\documentclass[11pt]{article}
\usepackage[a4paper, margin=1in]{geometry}
\usepackage{graphicx} % Required for inserting images
\usepackage{hyperref}
\usepackage{amsmath}
\usepackage{amssymb}
\usepackage{amsthm}
\usepackage{amssymb} 
\usepackage{float}
\usepackage{color}
\usepackage{xcolor}
\usepackage{tabularx}
\usepackage{array}
\usepackage{enumitem}
\usepackage{cleveref} 
\usepackage{algorithm}
\usepackage{mathtools}
\usepackage{stmaryrd}
\usepackage{comment}
\usepackage{booktabs}
\usepackage[normalem]{ulem}

\usepackage[font=small,labelfont=bf,tableposition=top]{caption}
\usepackage[font=footnotesize]{subcaption}
\usepackage[noend]{algorithmic}
\usepackage{mathtools,nccmath}
\newcommand{\mbf}[1]{\mathbf{#1}}

\newcommand*{\cA}{\mathcal{A}}
\newcommand*{\cB}{\mathcal{B}}

\newcommand*{\cE}{\mathcal{E}}
\newcommand*{\cF}{\mathcal{F}}

\newcommand*{\cS}{\mathcal{S}}

\newcommand*{\cU}{\mathcal{U}}

\newcommand*{\RR}{\mathbb{R}}

\newcommand*{\NN}{\mathbb{N}}

\newcommand*{\bD}{\mathbf{D}}
\newcommand*{\bE}{\mathbf{E}}

\newcommand*{\bH}{\mathbf{H}}

\newcommand*{\bI}{\mathbf{I}}

\newcommand*{\bP}{\mathbf{P}}

\newcommand{\crit}{\mathrm{crit }}
\newcommand{\vcrit}{\mathrm{vcrit }}

\makeatletter
\newcommand{\tpitchfork}{%
  \vbox{
    \baselineskip\z@skip
    \lineskip-.52ex
    \lineskiplimit\maxdimen
    \m@th
    \ialign{##\crcr\hidewidth\smash{$-$}\hidewidth\crcr$\pitchfork$\crcr}
  }%
}
\makeatother

\newcommand{\critg}{\mathrm{crit}_1 g}
\newcommand{\Argminlocg}{\mathrm{argmin{\text -}loc}_1 g}

\newcommand{\argmin}{\mathrm{argmin}}
\newcommand{\argminloc}{{\rm argmin{\text -}loc}}

\newcommand{\innerproduct}[2]{\langle #1, #2\rangle}

\newcommand{\firstvar}{x}
\newcommand{\secondvar}{y}
\newcommand{\thirdvar}{z}
\newcommand{\grad}{\nabla}

\newcommand{\jac}{\mathrm{Jac}\ }
\newcommand{\jacnospace}{\mathrm{Jac}}
\newcommand{\graph}{\mathrm{graph}\,}
\newcommand{\algo}{\mathcal{A}}

\newcommand{\partialss}{\partial_{\secondvar\secondvar}}
\newcommand{\partialsf}{\partial_{\secondvar\firstvar}}

\newcommand{\prox}{\mathrm{prox}}

\newcommand{\dist}{\mathrm{dist}}
\newcommand{\fn}[1]{y^{(#1)}}
\newcommand{\fni}{y^{(i)}}

\theoremstyle{plain}
\newtheorem{theorem}{Theorem}[section]
\newtheorem{assumption}{Assumption}[section]
\newtheorem{lemma}[theorem]{Lemma}

\newtheorem{corollary}[theorem]{Corollary}
\newtheorem{proposition}[theorem]{Proposition}

\newtheorem{taggedassumptionx}{Assumption}
\newenvironment{taggedassumption}[1]
 {\renewcommand\thetaggedassumptionx{#1}\taggedassumptionx}
 {\endtaggedassumptionx}

\newenvironment{taggedalgorithm}[1]
 {\renewcommand\thealgorithm{#1}\algorithm}
 {\endalgorithm}

\newtheorem*{theorem*}{Theorem}
\newtheorem*{proposition*}{Proposition}

\theoremstyle{definition}
\newtheorem{definition}[theorem]{Definition}
\newtheorem{remark}[theorem]{Remark}
\newtheorem{example}[theorem]{Example}

\newtheorem*{definition*}{Definition}

\newcommand{\TL}[1]{\textcolor{black}{#1}}

\title{Bilevel gradient methods and\\ the Morse parametric qualification condition}
\author{Jérôme Bolte\thanks{Toulouse School of Economics, University of Toulouse Capitole, Toulouse, France.} \and Tùng Lê \footnotemark[1] \and Edouard Pauwels\footnotemark[1] \and  Samuel Vaiter \thanks{CNRS \& Université Côte d'Azur, Laboratoire J. A. Dieudonné. Nice, France.}}
\date{}

\begin{document}

\maketitle

 \begin{abstract} We introduce the {\em Morse parametric qualification condition} for bilevel programming. Generic semi-algebraic functions are Morse parametric in a piecewise sense. Thus, bilevel programs with a Morse parametric lower level constitute a relevant intermediate class between strongly convex and fully generic lower levels. In this framework, we study  bilevel gradient algorithms with two strategies: the {\em single-step multi-step strategy}, which involves a sequence of steps on the lower-level problems followed by one step on the upper-level problem, and a \emph{differentiable programming strategy} that optimizes a smooth approximation of the bilevel problem. While the first is shown to be a biased gradient method on the problem with rich properties, the second, inspired by meta-learning applications, is less stable but offers simplicity and ease of implementation.
\end{abstract}
\sloppy
\section{Introduction}
Bilevel optimization provides a versatile formalism which encompasses a diversity of application settings \cite{dempe2002foundations,dempe2020bilevel}.
It has recently attracted significant attention in machine learning as it allows to formalize a broad spectrum of problem situations, including hyper-parameter tuning~\cite{pmlr-v48-pedregosa16}, meta-learning \cite{originalmaml,Franceschi2018BilevelPF}, data augmentation~\cite{rommel2022cadda}, deep equilibrium networks~\cite{bai2019deq} or neural architecture search~\cite{liu2019darts}. The present work investigates solution methods for bilevel optimization. We focus on bilevel gradient methods: first-order algorithms combining gradient steps for both levels of the problem. 

We consider bilevel optimization problems of the following form
\begin{equation}
    \label{eq:original-bilevel-optim}
    \tag{BL}
        \underset{\firstvar \in \RR^n,\secondvar \in \RR^m}{\min}  \quad f(\firstvar,\secondvar)\qquad \qquad \text{s.t.} \quad \secondvar \in  {\argmin}\; g(\firstvar,\cdot).
\end{equation}
where $f,g \colon \RR^n \times \RR^m \to \RR$ are respectively the {\em upper level objective} and the {\em lower level objective}. Throughout this paper, we assume that the lower level minimum is achieved.

\subsection{The Morse qualification condition and two algorithms}

\paragraph{Parametric Morse qualification condition:} A delicate aspect of bilevel programming is the gap between qualification regimes: it jumps from very strong conditions, when the lower level is strongly convex, to a great generality which must be handled by rather complex KKT machineries, see e.g., \cite{Allende2012SolvingBP,dempe2002foundations,dempe2020bilevel}. In the first case, applications are restrictive, in particular for machine learning, whereas in the  general case one needs to resort to complex qualification conditions to rule out pathologies even in the polynomial case \cite{bolte2024geometric}\footnote{The value function is typically discontinuous; any closed set can be expressed with smooth bilevel constraints; and polynomial bilevel optimization is $\Sigma_2^p$-hard, above NP in the polynomial hierarchy. For example, there is no numerically efficient certificate of stationarity or approximate stationarity for general bilevel problems.}. To partly remedy this difficulty, we introduce and study \emph{parametric Morse qualification conditions} (see \Cref{assumption:constraint-qualification}) for the lower-level function $g$. 
An intuitive picture is that as the parameters $x$ move, the lower-level landscape stays essentially the same, with an invariant ``Morse profile": the number and type of critical points remain constant, and each one traces a smooth branch as the parameter varies.

A first obvious interest is that this condition provides a meaningful intermediate between strongly convex and general nonconvex lower-level problems. There are moreover two other features that makes the setting both appealing and workable: (i) generically, a semi-algebraic lower-level objective \(g\) satisfies a \emph{piecewise} parametric Morse property (\Cref{prop:generic-ae-morse-parametric}); (ii) under this qualification condition, the lower-level critical and local-minima sets split into a finite union of \(C^2\) manifolds (\Cref{prop:crit-local-structures}), as depicted in \Cref{fig:crit-local-structures}. 

A consequence of the latter is that we obtain the following mixed-integer nonlinear programming relaxation for the bilevel problem (see also \Cref{tab:equivalence-two-relaxations}):
\begin{equation*}
	\begin{tabular}{lclcl}
		$\min f(\firstvar,\secondvar)$& $\xrightarrow{\text{relaxation}}$&$\min f(x,y)$&$\xleftrightarrow{\text{equivalent}}$&$\min f(x,y^{(i)}(x))$\\
		$\secondvar \in  {\argmin}\; g(\firstvar,\cdot)$&&$\secondvar \in  \argminloc\; g(\firstvar,\cdot)$&&$i \in \{1,\ldots,N\}$
	\end{tabular}
\end{equation*}
where $y^{(i)}$ are $C^2$ functions, $i = 1,\ldots, N$, highlighting the disconnected nature of the constraints in bilevel programming.   
Let us draw the attention of the reader on the fact that the relation $y = y^{(i)}(x)$ is implicitly defined through a nonconvex inner minimization, which must be approximated algorithmically, and the integer $N$ is typically unknown. Throughout this paper, we approximate lower-level solutions using a simple gradient-descent scheme (\Cref{algo:inner_gradient_descent}). 
Our main results (cf. \Cref{theorem:convergence-property-alternate-method,th:stabilityLocalMin,th:escapeInfinitySharpness}) relate to this \emph{local minima relaxation} (see \Cref{tab:equivalence-two-relaxations}).

\begin{figure}[t]
    \centering
    \includegraphics[width=0.8\linewidth]{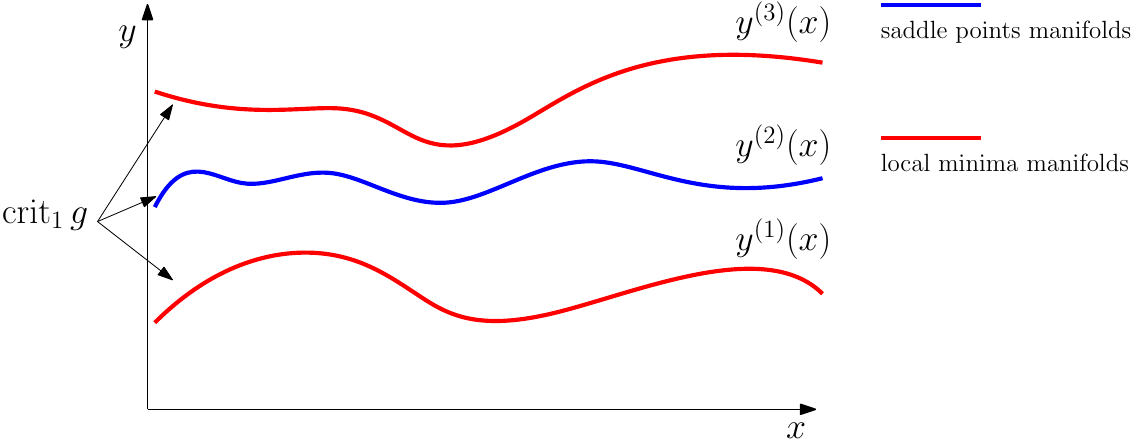}
    \caption{Structures of $\critg$ (see the definition in \S\ref{subsec:bilevel-problems}) under \Cref{assumption:constraint-qualification}; see also \Cref{prop:crit-local-structures}. The decomposition of the sets of critical points and local minima of $g$ as a finite union of manifolds is represented by several branches $y^{(1)}(x), y^{(2)}(x), y^{(3)}(x)$. Saddle-point manifolds (blue) and local-minima manifolds (red) illustrate the stratification of the critical set.}
    \label{fig:crit-local-structures}
\end{figure}

\paragraph{Convergence of a single-step multi-step bilevel gradient method:} Our first algorithmic result analyzes a simple alternating scheme: one outer step followed by multiple inner steps that approximately minimize the lower-level problem (see \Cref{algo:altenate-method}). The scheme can be recast as an \emph{inexact} method on the upper-level value function. Such strategies are common in bilevel gradient methods, with notable extensions \cite{arbel:hal-03869097,Ghadimi2018ApproximationMF,chen2021closing,Dagreou2022SABA,Dagreou2023SRBA,ji2024lowerbound,merchav2024fastalgorithmconvexcomposite}. 

Our analysis departs from previous work in two ways: (i) we allow general nonconvex lower-level problems with a possibly discontinuous argmin set-valued map; (ii) our assumptions are explicit and verifiable a priori.
In \Cref{theorem:convergence-property-alternate-method}, under the parametric Morse qualification condition and natural regularity and stepsize conditions, we prove that the method provides approximate solutions of the bilevel problem. This strengthens \cite{arbel:hal-03869097} in several directions and leverages the theory of inexact gradient methods for semi-algebraic optimization \cite{bolte2024inexact}.

\paragraph{Pseudo stability of a differentiable programming bilevel gradient method:}
Motivated by the analysis of Model-Agnostic Meta-Learning (MAML) models \cite{originalmaml,implicitMAML}, we consider a fully differentiable programming approximation of the lower level constraints in \Cref{algo:diagonal-method}, where the initialization of the lower level problem is treated as an upper level parameter. Such approaches were promoted in several works \cite{Franceschi2018BilevelPF,theory-multi-step,ji2020convergence,ji2022will,liu2021towards}. Although conceptually simple, the resulting minimization problems exhibit a fuzzy connection with the original bilevel problem. First, the differentiable programming approximation essentially erases the bilevel constraint as its stationary points are the same as the unconstrained upper level, see \Cref{prop:equivalence-single-level}. This negative result is balanced by a stability analysis that may explain the method’s partial empirical success. Indeed, well-qualified bilevel solutions  are pseudo-stable in the sense that, if the algorithm is to leave them, it generally needs many iterations to exit their neighborhood (\Cref{th:stabilityLocalMin}). Moreover, unconstrained local minimizers, which are irrelevant to the bilevel program, are, in practice, hard to reach (\Cref{th:escapeInfinitySharpness}). Taken together, these two facts are rather positive and point in the same direction. Yet  the picture is imperfect and instabilities can still arise, consistent with meta-learning evidence \cite{antoniou2018how}. This suggests either adjusting differentiable-programming pipelines with appropriate safeguards or modifications or turning to classical  gradient methods (as in the previous approach).

\subsection{Related work} 

Gradient methods and steepest descent methods are natural candidate solutions for bilevel optimization. Early works on the topic include differentiation of second-order local minima \cite{kolstad1990derivative}, bilevel optimality conditions, and the design of steepest descent methods \cite{savard1994steepest}. These studies focus on unique, well-posed lower level solutions, a simplifying hypothesis which allows to reduce the bilevel problem to that of the minimization of a locally Lipschitz objective for which general purpose methods could be considered \cite{pang1991minimization}. Overviews of bilevel programming contributions, prior to the latter machine learning developments, can be found in \cite{vicente1994bilevel,colson2007overview}.

The motivations for machine learning applications \cite{pmlr-v48-pedregosa16,Franceschi2018BilevelPF,lorraine2020MillionsofHyperparameters} have triggered the investigation of dedicated solution methods, the vast majority of them being gradient techniques \cite{arbel:hal-03869097,Ghadimi2018ApproximationMF,chen2021closing,Dagreou2022SABA,Dagreou2023SRBA,ji2024lowerbound}, most of them with a stochastic flavour. These more recent works have specificities compared to the broader bilevel optimization literature. First, the proposed theoretical analyses take explicitly into account the fact that the lower level is solved by an algorithm, hence producing only an approximate solution to the lower level problem. Second, the fact that the lower level solution algorithm is explicitly considered allows us to leverage all variations of algorithmic differentiation \cite{griewank2008evaluating} to obtain not only an approximate solution to the lower level, but also its derivative with respect to the upper level parameters, with limited overheads. This concept is known as the cheap gradient principle and the overall algorithmic differentiation framework sometimes called \emph{differentiable programming}. This represents a conceptual advantage and has a lot of practical benefits. The bilevel gradient algorithms which we study follow these principles: the lower level solution algorithm is explicit and we consider that it is possible to propagate derivatives through the algorithm. Another specificity of this literature is the possibility of considering the initialization of a lower level gradient scheme as an upper level parameter to be optimized. 
This is common in variants of Model-Agnostic Meta-Learning (MAML) models~\cite{originalmaml,implicitMAML} and was explicitly considered by several authors  \cite{Franceschi2018BilevelPF,theory-multi-step,ji2020convergence,ji2022will,liu2021towards}. We will see in the study of our differentiable programming bilevel gradient method that this has benefits, but also drawbacks.

In terms of theoretical guarantees, a majority of existing studies make simplifying assumptions ensuring the uniqueness and smoothness of the lower level solution, such as strong convexity. The bilevel problem then reduces to a smooth or at least Lipschitz composite optimization problem \cite{pang1991minimization,ye1995optimality,Mordukhovich2020}. Beyond strong convexity, there is no complete general treatment of nonconvex lower level problems, with possibly multiple lower level solutions.
A specificity of our analysis is the generality of the lower-level problem, as \textbf{we do not assume neither convexity nor the existence of a unique minimizer}. This lack of structure significantly makes algorithm design and theoretical guarantees harder. Some approaches introduce auxiliary initializations to guide optimization \cite{liu2021towards}, while others rely on interior-point methods to mitigate computational costs \cite{liu2021valuefunction}. \TL{Certain works shift their focus on finding approximated solutions and combine the single-level penalized functions uniform growth conditions on the lower level problem to establish convergence properties \cite{shen2023penalty}.} The ambiguity arising from multiple lower-level solutions can be addressed using selection maps (to allow to prove an implicit function theorem even in degenerate cases) \cite{arbel:hal-03869097}. Reformulating bilevel problems as single-level ones introduces nonsmooth constraints requiring specialized algorithms \cite{guihua2014solving}. From an analytical perspective, variational analysis and generalized differentiation provide optimality conditions when classical constraint qualifications fail \cite{Mordukhovich2020}, often requiring calmness conditions \cite{ye1995optimality} and sensitivity analysis to properly characterize solutions \cite{dempe2012sensitivity}. The assumption of a unique lower-level solution is typically unrealistic and some work~\cite{liu2020generic} propose to reformulate bilevel optimization as an alternative kind of optimization. Finally, bilevel optimization is provably inherently hard, with bilevel smooth problems being as difficult as general lower semicontinuous minimization and bilevel polynomial problems reaching complexity beyond NP-hard problems \cite{bolte2024geometric}.

\section{Presentation of the problem and main algorithms}

\subsection{The bilevel problem and the lower-level gradient routine}\label{subsec:bilevel-problems}

Given $\firstvar$ in $\RR^n$, the global minimizer of $g_x:=g(\firstvar,\cdot)$ does not have a closed-form expression and one often resorts to algorithms to solve \emph{approximately} the lower level problem. In this work, we assume to have access to an iterative algorithm. Given $\firstvar \in \RR^n$, one initializes $\secondvar_0 = z$ and run an iterative method $y_{i+1} = \algo(x,y_i)$ for a fixed number of iterations $i = 1,\ldots, k$. 
Here, our algorithm $\algo$ is chosen to be the classical Cauchy's gradient descent (GD) algorithm with a fixed stepsize. Denoting by $\algo^k$, $k$ iterations of the algorithm, we consider $y = \algo^k(x,z)$ where $z$ is the initialization of the algorithm, see \Cref{algo:inner_gradient_descent}. 

\begin{taggedalgorithm}{LLD}[H]
	\centering
	\caption{Lower level descent $\algo^k$ ($k$-step fixed-stepsize gradient descent)} 
	\label{algo:inner_gradient_descent}
	\begin{algorithmic}[1]
        \REQUIRE $\firstvar \in \RR^n, \thirdvar \in \RR^m, \alpha_g > 0, k \in \NN$. 
		\STATE $\secondvar_0 \gets \thirdvar$.
        \FOR{$j = 1, \ldots, k$}
            \STATE $\secondvar_{j} \gets \secondvar_{j-1} - \alpha_g \grad_\secondvar g(\firstvar, \secondvar_{j - 1})$ \hfill $\rhd$\COMMENT{gradient descent update}
        \ENDFOR
        \RETURN $y_k$
	\end{algorithmic}
\end{taggedalgorithm}

We will work under the following assumptions which ensure that for any $k \in \NN$, $\algo^k$ is a continuously differentiable mapping. 
\begin{assumption}[Regularity of the upper and lower functions]
Let $m,n > 0$.\\
(i) The functions $f,g \colon \RR^n \times \RR^m \to \RR$ are $C^2$ and $C^3$ respectively, and semi-algebraic. \\
(ii) For all $x \in \RR^n$, $g_x$ has a $L_g$-Lipschitz contiuous gradient and is coercive.\\ 
 (iii) The lower level step size parameter $\alpha_g$ satisfies $0 < \alpha_g L_g < 1$.
	\label{assumption:lowerLevel}
\end{assumption}

Since $g_x(\cdot) = g(x, \cdot)$ is not assumed to be convex,  \Cref{algo:inner_gradient_descent} merely approaches critical points of $g_x$. On the other hand, it is known that, generically the algorithm finds local minima \cite{pemantle1990nonconvergence,goudou2009gradient}.

This leads us to consider two fundamental relaxations of our original bilevel program: {\em the critical set relaxation} and {\em the local-min relaxation}. The latter is already considered and studied with the so-called Karush-Kuhn-Tucker (KKT) approach \cite{Allende2012SolvingBP}.

Denote by $\critg$ the (set-valued) mapping $x \rightrightarrows \crit\, g_x $, and $\Argminlocg$ the (set-valued) mapping $x\rightrightarrows \argminloc\, g_x$, our problems can actually be rewritten:
\begin{table}[h]
    \centering
    \begin{tabular}{>{\centering\arraybackslash}m{0.43\textwidth}|>{\centering\arraybackslash}m{0.5\textwidth}}
         \toprule
         Critical points relaxation & Local minima relaxation \\
         \midrule
         \begin{minipage}{0.43\textwidth}
            \begin{equation}
            \label{eq:crit-bilevel-optim}
            \tag{BL-crit}
            \begin{aligned}
                \min  \, f(\firstvar,\secondvar)\\
               (\firstvar,\secondvar) \in \,  \graph \critg
            \end{aligned}
            \end{equation}
        \end{minipage} & 
        \begin{minipage}{0.5\textwidth}
            \begin{equation}
            \label{eq:loc-bilevel-optim}
            \tag{BL-loc}
            \begin{aligned}
                \min f(\firstvar,\secondvar)\\
                (\firstvar,\secondvar) \in  \graph \Argminlocg
            \end{aligned}
            \end{equation}
        \end{minipage} \\
        \bottomrule
    \end{tabular}
\end{table}

\subsection{Two bilevel gradient algorithms}

We consider two bilevel gradient methods with different strategies.
For a given $k \in \NN$, we consider the approximate value function objective,
\begin{align}
    \tag{DP-BL}
    \label{eq:algorithmic-bilevel}
	\varphi^k(\firstvar,\thirdvar) := f(\firstvar, \algo^k(\firstvar,\thirdvar))
\end{align}
which may be seen as an unconstrained smooth optimization problem approximating \eqref{eq:original-bilevel-optim}. 
Indeed, due to the convergence properties of the gradient method for semi-algebraic functions \cite{attouch2013Convergence}, under \Cref{assumption:lowerLevel}, the corresponding function has the property that the limit below exists and defines a function $\varphi^\infty$ as follows:
\begin{align*}
	\lim_{k\to \infty} \varphi^k(x,z)&:=\varphi^\infty(x,z) \\
	\displaystyle \graph \varphi^\infty&=\left\{\left((x,z),f(x,y(x,z))\right):\TL{y(x,z) = \lim_{k \to \infty} \cA^k(x,z)\in \crit g_x}\right\}.
\end{align*}

\paragraph{Single-step Multi-step strategy.} We first consider a two-stage method: one approximates the lower-level critical points (or local minima) for a fixed $x$, in our case, by \Cref{algo:inner_gradient_descent} with $\cA^k$, then 
performs a gradient step on the upper level. This is what we call, in this work, the {\em Single-step Multi-step} strategy.

\begin{taggedalgorithm}{SMBG}[H]
	\centering
	\caption{Single-step Multi-step bilevel Gradient Algorithm} 
	   \label{algo:altenate-method}
	   \begin{algorithmic}[1]
        \REQUIRE $\alpha_f > 0, \alpha_g > 0, k \in \NN, L \in \NN$. 
		\STATE Initialize $(x_0, y_0)$
        \FOR{$\ell = 1, \ldots, L$}
            \STATE \textcolor{black}{$y_\ell \gets \cA^k(x_{\ell - 1}, y_{\ell - 1})$} \hfill $\rhd$\COMMENT{$k$ GD steps for lower-level $g(x_{\ell - 1}, \cdot)$}
            \STATE $x_{\ell} \gets x_{\ell - 1} - \alpha_f \nabla_x \varphi^k(x_{\ell - 1}, y_{\ell - 1}))$ \hfill $\rhd$\COMMENT{one GD step for upper-level}
        \ENDFOR
        \RETURN $(x_L, y_L)$
	\end{algorithmic}
\end{taggedalgorithm}

The alternation of update steps between each level is classical in bilevel programming. Variations close to the proposed description were explored recently motivated by machine learning applications \cite{arbel:hal-03869097,Ghadimi2018ApproximationMF,chen2021closing,Dagreou2022SABA,Dagreou2023SRBA,ji2024lowerbound}. These works, however, harness the implicit function theorem and approximate implicit differentiation (AID), we consider the most basic version of this idea.

The discrepancy between the type of update and number of iterations dedicated to the two levels is natural: in the first place the lower level constraint must be (approximately) enforced, then and only then, progress can be made on the upper level.

\paragraph{Differentiable programming strategy}
Since $\varphi^k$ approximates $\varphi^\infty$ which corresponds to the critical point relaxation \eqref{eq:crit-bilevel-optim} of problem \eqref{eq:original-bilevel-optim}, it is tempting to minimize directly the unconstrained, smooth, $\varphi^k$ for a given fixed $k$. 
This is the main idea of the {\em differentiable programming} strategy. 
\begin{taggedalgorithm}{DPBG}[H]	
    \centering
	\caption{Differentiable Programming bilevel Gradient Method} 
	   \label{algo:diagonal-method}
	   \begin{algorithmic}[1]
        \REQUIRE $\alpha_f > 0, \alpha_g > 0, k \in \NN, L \in \NN$. 
		\STATE Initialize $(x_0, z_0)$.
        \FOR{$\ell = 1, \ldots, L$}
            \STATE $\displaystyle \begin{pmatrix}
                \textcolor{black}{z_{\ell}} \\ x_{\ell}
            \end{pmatrix} = \begin{pmatrix}
                \textcolor{black}{z_{\ell-1} - \alpha_f \nabla_z \varphi^k (x_{\ell - 1}, z_{\ell - 1})} \\ x_{\ell-1} - \alpha_f \nabla_x \varphi^k (x_{\ell - 1}, z_{\ell - 1})
            \end{pmatrix}$ 
        \ENDFOR
        \RETURN $(x_L, z_L)$
	\end{algorithmic}
\end{taggedalgorithm}

Such strategies are widespread in MAML~\cite{originalmaml} which formulate bilevel problems to find a common algorithmic initialization, good for several learning tasks. They combine the minimization of the smooth approximation $\varphi^k$ \cite{Franceschi2018BilevelPF} with the idea that derivatives of $\varphi^k$ are accessible at a moderate cost using algorithmic differentiation \cite{maclaurin2015gradient,mehmood2020automatic,grazzi2020iteration}. In a broader context, the idea of minimizing the surrogate $\varphi^k$ jointly in $x$ and $z$ appears in \cite[Algorithm 1]{liu2021towards}. In particular, they argue that when the lower-level problem $g_x$ admits multiple optimal solutions (or critical points), it is necessary to also update $z$ w.r.t.~$\varphi^k$ to avoid initializing in a ``bad'' region, resulting in a sub-optimal $y^\star(x)$ (remind that we want to minimize $f$ w.r.t. both $x$ and $y$ in \eqref{eq:original-bilevel-optim}).

\begin{remark}[Machine learning perspective]
	For both algorithms, the gradient estimation for $\algo^k$ (cf. \Cref{algo:inner_gradient_descent}) and its compositions is called iterative differentiation (ITD) in the context in machine learning. This relies on an intensive usage of algorithmic differentiation routines \cite{gilbert1992automatic}. In addition,  \Cref{algo:altenate-method} and \Cref{algo:diagonal-method} combine $k$ gradient steps on the inner problem with an algorithmic step on the outer problem, a feature which is described as a {\em double loop} strategy in the machine learning literature. This literature focuses on complexity estimates and stochastic approximation aspects. In contrast, the main specificity of our work is to consider a simpler algorithmic setting, but a broader class of problems with a focus on the difficulties caused by the absence of lower level convexity.
\end{remark}

\section{Morse-parametric constraint qualifications and the lower level gradient method}
\subsection{Morse-parametric constraint qualifications} Let us now consider the fundamental subject of constraint qualification. To establish necessary conditions for local minimizers, works on this question \cite{ye1995optimality,dempe2012sensitivity,Mordukhovich2020,hendrion2011calmness} requires various assumptions such as the so-called \emph{partial calmness} \cite{ye1995optimality} and the inner-semicontinuity of the set-valued mappings related to the mapping $x \rightrightarrows \argmin \,g_x$ \cite{dempe2002foundations,Mordukhovich2020,hendrion2011calmness}. Despite the mathematical insights of previous works, finding a class of optimization problems that satisfy these requirements is non-trivial.

\smallskip

In this section, we introduce a family of nonconvex bilevel problems whose constraints qualification is much simpler. In particular, we require the lower-level function $g$ to satisfy the so-called \emph{parametric Morse} property, whose definition is given below. 
\begin{definition}[Morse and parametric Morse functions]
    \label{def:parametric_morse}
    Let $\phi: \RR^m \to \RR$ be a $C^2$ function. We say that $\phi$ is a \emph{Morse function} if its Hessian matrix $\nabla^2 \phi(\secondvar)$ is invertible for any critical point $y$ of $\phi$, i.e., $\nabla \phi(y) = 0$. Let $g: U \times \RR^m \to \RR$ be a $C^2$ function where $U \subset \RR^n$ is open and connected. The function $g$ is called \emph{parametric Morse} on $U$ if, for all $x \in U$, $g_x := g(x, \cdot)$ is a Morse function. If $U = \RR^n$, the function is simply called parametric Morse.
\end{definition}

This is a particular case of the Morse-Bott parametric property, proposed in \cite{arbel:hal-03869097}. Further variants in bilevel optimization and machine learning were proposed in \cite{chen2024finding,kwon2023penalty}. Here are two examples of parametric Morse functions:

\begin{example} 
    \label{example:morse-parametric}
    (a) \emph{Strongly convex functions}: 
    If $g_x$ is a $C^2$ and $\mu$-strongly convex for all $\firstvar \in \RR^n$, then $g$ is parametric Morse since $\partialss f(\firstvar,\secondvar)  \succeq \mu \bI$. This class is usually assumed in many works on lower-level convex bilevel optimization such as \cite{Dagreou2022SABA,Dagreou2023SRBA}.\\
(b) \emph{Composition of Morse functions and diffeomorphisms}: 
If $g = F \circ h$ where $F: \RR^m \to \RR$ is a $C^2$, Morse function and $h: \RR^n \times \RR^m \to \RR^m$ is $C^2$, each mapping $h(\firstvar, \cdot)$ is a diffeomorphism from $\RR^m$ to itself, then $g$ is parametric Morse.
\end{example}

We introduce the constraint qualification assumptions for the bilevel optimization in \ref{assumption:constraint-qualification}. 
\begin{taggedassumption}{Morse QC}[Morse-parametric qualification conditions]
    \label{assumption:constraint-qualification}
    The function $g$ is parametric Morse and the set-valued mapping $\critg: x \rightrightarrows \crit\, g_x $ is non-empty and locally bounded.
\end{taggedassumption}

\subsection{Relevance of the parametric Morse property}
\label{sec:genericity-ae-morse-parametric}
In mathematical programming, constraint qualifications have two distinctive features. First, they allow to obtain an algebraic description of problem solutions, and second, they hold true for the typical, or generic, problems of the class. For example, Slater's interiority condition for convex programs can always be obtained by a slight perturbation of inequality constraints and is sufficient to ensure that all solutions are described using KKT conditions. This has numerous extensions in nonlinear programming, see for example \cite{bolte2018qualification} and references therein. In the bilevel programming context, qualification conditions satisfying these two qualitative features are missing.

We will see that, while \Cref{assumption:constraint-qualification} allows for a useful description of problem solutions, parametric Morse functions are not dense in the class of $C^2$ functions, and hence cannot be considered as typical in the $C^2$ topology (see \Cref{ex:morseParamNongeneric}). 
However, we argue that the parametric Morse property possibly has a central role in bilevel optimization. Indeed, a generic semi-algebraic bilevel problem satisfies the following {\em piecewise parametric Morse} property. 

\label{sec:generic-morse-parametric-ae}
\begin{definition}[Piecewise parametric Morse functions]
    \label{def:ae-parametric-morse}
    A $C^2$ function $g: \RR^n \times \RR^m \to \RR$ is \emph{piecewise parametric Morse} if $g_x := g(x, \cdot)$ is Morse, for every $x$ in an open dense set.
\end{definition}
The main differences with parametric Morse property are 1/ that the property is required to hold on a dense {and open}
subset, and 2/ this set is not required to be connected.
The piecewise parametric Morse property turns out to be generic for semi-algebraic functions.
\begin{proposition}[Genericity of piecewise parametric Morse functions]
    \label{prop:generic-ae-morse-parametric}
	Given a $C^2$, semi-algebraic function $g: \RR^n \times \RR^m \to \RR$, the set of vectors $a \in \RR^m$ such that $g(x,y) - \innerproduct{a}{y}$ is piecewise parametric Morse is semi-algebraic and dense in $\RR^m$ (hence residual and of full measure).
\end{proposition}
Before proving the result, we recall that a semi-algebraic set is the finite union of solution sets of systems of polynomial inequalities. A semi-algebraic piecewise Morse parametric function is, therefore, Morse parametric in the sense of \Cref{def:parametric_morse} on each of the finitely many connected components of semi-algebraic dense {and open} set required by \Cref{def:ae-parametric-morse}. This result suggests that a general description of the solution set of a generic bilevel program requires a deep understanding of the simpler Morse parametric assumption. Therefore, this property appears as a relevant intermediate between widely studied strongly convex lower levels, to the most general nonconvex lower levels.
\begin{proof}[Proof of \Cref{prop:generic-ae-morse-parametric}]
	The fact that the resulting set is semi-algebraic is a direct consequence of the Tarski-Seidenberg principle, see for example \cite{coste2000introduction}.

    Define $A:= \{(x,z) \mid z = \nabla_y g(x,y) \text{ where } \det(\nabla^2_{yy} g(x,y)) = 0\}$. Since $g$ is semi-algebraic, so are $A$ and its fibers i.e., $A(x) = \{z \mid (x,z) \in A\}$. Moreover, for any $x$, $A(x)$ is also the set of critical values of the mapping $\nabla_y g_x = \nabla_y g(x, \cdot)$, hence of zero Lebesgue measure (thanks to Sard theorem). By Fubini-Tonelli theorem, the Lebesgue measure of $A$ is zero as well. Combine with the fact that $A$ is semi-algebraic, $A$ is of dimension at most $(m + n - 1)$.

    Consider the projection mapping (onto the last $m$ coordinate) $\pi : A \to \RR^m$. As a consequence of Hardt's triviality theorem \cite[Corollary 4.2]{coste2000introductionSA}, the set 
    \begin{equation*}
        P := \{z \in \RR^m \mid \dim \pi^{-1}(z) = n\}
    \end{equation*}
	is semi-algebraic of dimension at most $\dim A - \TL{n}  \leq \TL{m}-1$. This concludes the proof since for a generic $z \in \RR^m$, $z \not \in P$, therefore $\dim \pi^{-1}(z) <n$ and $g_x$ is Morse for any $x \not \in \pi^{-1}(z)$.
\end{proof}

Finally, we also illustrate why parametric Morse functions property are not dense among $C^2$ functions in the following example.
\begin{example}
	\label{ex:morseParamNongeneric}
    Consider the function $g: \RR^2 \to \RR: (x,y) \mapsto (x - y^2)^2$. This function is not Morse parametric since its critical sets is given by:
    \begin{equation*}
        \crit \, g_x := \begin{cases}
            0 & x \leq 0\\
            \{0, \pm \sqrt{x}\} & x > 0
        \end{cases}.
    \end{equation*}
    Hence, $g_x$ is not Morse at $x = 0$. Moreover, an arbitrarily $C^2$ small perturbation $e$ cannot make the function $g + e$ parametric Morse. 
    Here, the size of a $C^2$ perturbation $e$, can be considered as $\|e\|_{C^2} = \sup_{x,y} |e(x,y)| + \| \nabla e(x,y)\| + \|\nabla^2 e(x,y)\|$. Note that 
    \begin{align}
        \label{eq:ineqComponentsMors}
        \sup_{x,y} |e_x(y)| + |e_x'(y)| + |e_x''(y)| \leq \sup_{x,y} |e(x,y)| + \| \nabla e(x,y)\| + \|\nabla^2 e(x,y)\|,
    \end{align}
    since $\|\nabla e(x,y)\| \geq |e_x'(y)|$ and $\|\nabla^2 e(x,y)\| \geq |e_x''(y)|$ for all $x,y \in \RR$. In other words $\|e\|_{C^2} \geq \sup_x \|e_x\|_{C^2}$. Assume that $\|e\|_{C^2} < 1$.

    Returning to $g$, we have $g'_x(y) = 4y(y^2-x)$ and $g''_x(y) = 12y^2-4x$. It is clear that the critical points of $g_x + e_x$ are locally bounded as $\lim_{|y| \to + \infty} \inf_{x \in K} |g'_x(y)| = +\infty$ for any compact interval $K$. Let us show that the number of critical points of $g_x + e_x$ cannot be constant in $x$.

    For $x = -1$, it is clear that $g_{-1} + e_{-1}$ is strongly convex as $g_{-1}''(y) \geq 4$  and $|e_{-1}''(y)| \leq 1$
 for all $y$. So it has a single critical point which corresponds to its global minimum. Let us consider $x = 1$, we have the following variations for each interval $I$ as below, chosen to contain none or a single critical point of $g_1$
    \begin{center}
    \begin{tabular}{c|c|c|c|c|c|c|c}
        $I$ &$(-\infty, -2]$ & $\left[-2,\frac{-1}{\sqrt{2}}\right]$ &  $\left[\frac{-1}{\sqrt{2}}, \frac{-1}{2}\right]$&$\left[\frac{-1}{2},\frac{1}{2}\right]$ & $\left[\frac{1}{2}, \frac{1}{\sqrt{2}}\right]$ & $\left[\frac{1}{\sqrt{2}}, 2\right]$ & $[2,+\infty)$ \\[1.5ex]\hline&&&&&&&\\[-2ex]
         $g_1'(I)$& $\leq-24$&$\supset [-1,1]$ &$\geq 1$&$\supset [-1,1]$& $\leq -1$&$\supset [-1,1]$&$\geq 24$ 
    \end{tabular}     
    \end{center}
    From this table, we see that since $|e_1'(y)| <1$ for all $y$, for each $I$ as above such that $g_1$ has no critical point, $g_1 + e_1$ also has no critical point, and for each $I$ such that $g_1$ has one critical point, $g_1 + e_1$ has at least one critical point. We conclude that $g_1 + e_1$ has at least three critical points.

    Finally, the contradiction can be deduced from \Cref{lemma:extension-to-Rn} (see \Cref{section:structure-critical-points} for more detail): supposing that $g$ is parametric Morse, since for each critical point of $g_1$, one can find a smooth function defined globally on $\RR^2$ that gives the solution to $\nabla_y g(x,y) = 0$, $g_{-1}$ must have at least three distinct critical points (otherwise, these functions are identical due to the uniqueness stated in \Cref{lemma:extension-to-Rn}). This is, however, not possible due to our argument about $g_{-1}$.
\end{example}

\TL{All following results in \Cref{section:structure-critical-points,sec:lower-level-gd,sec:single-step-multi-step} only hold under the assumption of parametric Morse property (cf. \Cref{def:parametric_morse}), and not its piecewise version (cf. \Cref{def:ae-parametric-morse}). Nevertheless, as we argued earlier, due to the genericity of the piecewise parametric Morse property, the parametric Morse property becomes an important intermediate case that is worth studying. Moreover, given a piecewise parametric Morse function $g$, if one supposes that the iterates always stay in the same connected component of the set $x$ where $g_x$ is Morse, then all of our results transfer naturally. An important future extension is to understand better the case where the iterates can ``jump'' between the connected components and identify the conditions in which interesting theoretical results (such as convergence) can be established.   
}
\subsection{Structure of the crit and argmin-loc mappings under \Cref{assumption:constraint-qualification}}
\label{section:structure-critical-points}
The interest of \Cref{assumption:constraint-qualification} (and also \Cref{def:parametric_morse}) lies in the fact that the constraint $\TL{y \in \crit \,g(x, \cdot)}$ can be written equivalently as: $y \in \{\fni(x) \mid i = 1, \ldots, N\}$ where $N$ is a natural number and $\fni, i = 1, \ldots, N$ are $C^1$ functions. This representation can actually be made of constant index, we only distinguish local minima from other critical points.

\begin{proposition}[Critical points and local minima are finite union of manifolds]
    \label{prop:crit-local-structures}
	Under \Cref{assumption:constraint-qualification} and $g$ being semi-algebraic $C^k$ ($k > 1$), there exists integers $M\geq N \geq 0$, and $M$ semi-algebraic $C^{k-1}$ functions $\fn{1}, \ldots, \fn{N}, \ldots, \fn{M} \colon \RR^n \to \RR^m$ with distinct values at each points such that\footnote{We use the convention that for $N = 0$ or $M=0$ the union is empty}:
	 \begin{align}
        & \:\: \graph \critg = \bigcup_{i = 1}^M \graph \fn{i}. \label{lemma:function_splitting}\\
        &  \:\:\graph \Argminlocg = \bigcup_{i = 1}^N \graph \fn{i}. \label{lemma:splitting-function-local-minima}
    \end{align}
\end{proposition}
\begin{proof}
\textbf{Proof of \eqref{lemma:function_splitting}.}
The graph of $\critg$ is a semi-algebraic set (since $g$ is semi-algebraic) such that for each $x \in \RR^n$, the fibers $\crit \,g_x$, if non-empty, contain only isolated points thanks to the implicit function theorem. Hence, the connected components of $\crit g_x$ are singletons. By \cite[Properties 4.4]{o-minimal-structures}, the number of connected components of $\crit \, g_x$ is uniformly bounded above by an integer. 
Moreover, consider the following claim:
\begin{lemma}
    \label{lemma:extension-to-Rn}
    Suppose that \Cref{assumption:constraint-qualification} holds such that $g$ is $C^k$ ($k>1$), and let $(\bar{x},\bar{y})$ be a solution to the equation $\nabla_y g(\bar{x},\bar{y}) = 0$. There is a unique globally defined $C^{k-1}$ function $y \colon \RR^n \to \RR^m$ such that for any $x \in \RR^n$, $y(x) \in \RR^m$ is the unique solution to the equation $\nabla_y g(x,y) = 0$ locally, and $y(\bar{x}) = \bar{y}$.
\end{lemma}
    \Cref{lemma:extension-to-Rn} is proved at the end of this proof.
    Denote by $M(x)$ the size of $\crit \,g_x$, \Cref{lemma:extension-to-Rn} allows to conclude that this is a (finite) constant, i.e., $M(x) = M, \forall x \in \RR^n$. Indeed, each solution at a given $\bar{x}$ might be extended to globally defined solution, and if any two of such functions are equal at a given $x$, they are equal on the whole space.

    To finish the proof of \eqref{lemma:function_splitting}, it remains to show that each function $\fni, i =1, \ldots, M$ is semi-algebraic. This is because for each $i$, $\graph \fni$ is a connected component of the semi-algebraic set $\graph \critg$ (see \cite[Property 4.3]{o-minimal-structures}).

	\textbf{Proof of \eqref{lemma:splitting-function-local-minima}.} It is sufficient to prove that for each $i \in \{1, \ldots, M\}$, either $\fni(\firstvar)$ is a local minimum of $g_x$ for all $\firstvar \in \RR^n$ (positive definite Hessian), or $\fni(\firstvar)$ is a strict saddle point of $g_x$ for all $\firstvar$ (indefinite Hessian). Note that these are only two possibilities due to the parametric Morse assumption. 
 By the continuity of the eigenvalues and \ref{assumption:constraint-qualification}, we can conclude that $\nabla^2_{yy} g(\firstvar, \fn{i}(\firstvar))$ is either positive definite for all $\firstvar$ or indefinite for all $\firstvar$. In fact, the index, i.e., the number of positive eigenvalues, of $\nabla^2_{yy} g(\firstvar, \fn{i}(\firstvar))$ is constant for all $\firstvar$, as a change would result in a zero eigenvalue which contradicts the parametric Morse assumption. The result follows after permuting $\{1, \ldots, M\}$ so that local minima functions take the first $N$ positions.
\end{proof}
We consider an intermediate claim before proving \Cref{lemma:extension-to-Rn}.
    \begin{lemma}
        \label{lemma:twoBalls}
        Suppose that \Cref{assumption:constraint-qualification} holds.
        Let $B_1,B_2 \subset \RR^n$ be two open balls and $y_1 \colon B_1 \to \RR^m$ and $y_2 \colon B_2 \to \RR^m$ be continuous solutions to $\nabla_y g(x,y) = 0$. If $y_1(x) = y_2(x)$ for some $x \in B_1 \cap B_2$, then $y_1(x) = y_2(x)$ for all $x \in B_1 \cap B_2$.
    \end{lemma}
\begin{proof}[Proof of \Cref{lemma:twoBalls}]
    Set $\cE:= \{x \in B_1 \cap B_2 \mid y_1(x) = y_2(x)\}$. This set is closed (in the induced topology on $B_1 \cap B_2$) because it is the pre-image of $0$ by the continuous function $y_1 - y_2$. It is also open, thanks to the implicit function theorem. Since $B_1 \cap B_2$ is convex, hence path connected, we have either $\cE = \emptyset$ or $\cE = B_1 \cap B_2$. The conclusion follows because $\cE \neq \emptyset$ by assumption.
\end{proof}

\begin{proof}[Proof of \Cref{lemma:extension-to-Rn}]
    Define $R$ as: 
    \begin{equation}
        \label{eq:existence-of-function}
        R = \sup_{r > 0} \left\{r \mid\, \exists y_r \colon B(\bar{x},r) \to \RR^m,\, C^{k-1} \text{ s.t. }\forall x \in B(\bar{x},r),  \nabla_y g(x,y_r(x)) = 0,\, y_r(\bar{x}) = \bar{y}\right\} .
    \end{equation}

    We first argue that there exists a unique $C^{k-1}$ function $y: B(\bar{x},R) \to \RR^m$ such that $\nabla_y g(x, y(x)) = 0$ and $y(\bar{x}) = \bar{y}$. Indeed, by definition of $R$, for any $0 < a \leq b < R$, there exists two functions $y_{a}: B(\bar{x},a) \to \RR^m$ and $y_b: B(\bar{x},b) \to \RR^m$ satisfying the conditions of \Cref{eq:existence-of-function}. \Cref{lemma:twoBalls} ensures that $y_a(x) = y_b(x), \forall x \in B(\bar{x},a)$.  For any $x \in B(\bar{x},R)$, we set $y(x)=y_r(x)$ for any $r>0$ such that $\dist(\bar{x},x) < r < R$. This is well defined since, by the argument above the chosen value does not depend on $r$. In addition, $y$ is trivially $C^{k-1}$ (since $y_r$ is $C^{k-1}$ by the implicit function theorem), and unique by \Cref{lemma:twoBalls}.
    
    For the sake of contradiction, suppose that $R$ is finite. Using the boundedness assumption, for any $\tilde{x} \in \partial B(\bar{x},R)$ -- the boundary of $B(\bar{x},R)$, there is a solution $\tilde{y}$ which is an accumulation point of $y(x)$ as $x \to \tilde{x}$. By the implicit function theorem, there is an open ball neighborhood of $\tilde{x}$, $V_{\tilde{x}}$ and a $C^{k-1}$ function $f: V_{\tilde{x}} \to \RR^m$ such that $\TL{\nabla_y} g(x,f(x)) = 0, \forall x \in V_{\tilde{x}}$, and $f(\tilde{x}) = \tilde{y}$. Since the boundary is compact, it can be covered by finitely many such open balls, says $B_i$, $i=1,\ldots, K$ with centers on the sphere. In each open ball $B_i$, let $x_i$ be its center and $f_i: B_i \to \RR^m$ be the $C^{k-1}$ function satisfying $f_i(x_i) = y_i$. Denote $B(\bar{x},R)$ by $B_0$ and $f_0:= y$, we extend the function $y$ to the function $\tilde{y}$ of a larger domain $\cB: = \bigcup_{i=0}^K B_i$ as follows:
   \begin{equation}
        \label{eq:defExtension}
       \tilde{y}(x) = f_i(x) \text{ if } x \in B_i.
   \end{equation}
   A contradiction results from the following two claims:
   \begin{enumerate}
       \item $\tilde{y}(x)$ in \eqref{eq:defExtension} is well defined for all $x \in \cB$, and the resulting function is $C^{k-1}$.
       \item There exists $\delta > 0$ such that $B(\bar{x}, R + \delta) \subseteq \cB$.
   \end{enumerate}
   In particular, the second claim contradicts the fact that $R<+\infty$ is the maximum, concluding the proof.
   \begin{enumerate}[leftmargin=*]
       \item The well-definedness of $\tilde{y}$: we need to show that the value in \eqref{eq:defExtension} does not depend on $i$. It is sufficient to prove that for any $0 \leq i < j \leq k$, $f_i(x) = f_j(x)$, for all $x \in B_i \cap B_j$. Fix any $i<j$ such that $B_i \cap B_j \neq \emptyset$, by \Cref{lemma:twoBalls}, it suffices to show that the set $\{x \in B_i \cap B_j \mid f_i(x) = f_j(x)\}$ is non-empty. Indeed if $i = 0$, the implicit function theorem ensures local uniqueness of solutions to $\nabla_y g(x,y) = 0$ around $(x_j,y_j)$, and $y_j$ is a point of accumulation of $f_0$ on $B_0$, so that we have $f_i(x) = f_j(x)$ for some $x \in B_0 \cap B_j$ close to $x_j$. If both $i, j > 0$, since $x_i, x_j$ belong to $\partial B_0$, there exists $x \in B_i \cap B_j \cap B_0$ and, by the argument for $i = 0$, we have $f_i(x) = f_0(x) = f_j(x)$.
       \item The existence of $\delta$: we claim that $\delta := \dist(\overline{B(\bar{x},R)}, \cB^C) > 0$. Indeed, this is the distance between a compact ball i.e., $\overline{B(\bar{x},R)}$) and a closed set (the complement of the open set $\cB$), so the infimum is attained. If $\delta = 0$, there is a point on the sphere not belonging to any $B_i$, which contradicts the fact that we have a covering. Hence $B(\bar{x}, R + \delta) \subset \cB$.
   \end{enumerate}
\end{proof}

\begin{remark}
    \label{remark:uniform-boundedness}
    If in \Cref{prop:crit-local-structures}, one replaces the condition $g$ is semi-algebraic by the uniform boundedness of the cardinality of $\crit \,g_x$, the conclusions of \Cref{prop:crit-local-structures} still hold, but $\fni$ is not necessarily semi-algebraic. 
\end{remark}

\Cref{fig:crit-local-structures} illustrates the claims in \Cref{prop:crit-local-structures}. We remark that the same claim for \emph{global} minimizers of $g_x$ is not true in general. Indeed, under \Cref{assumption:constraint-qualification}, $\argmin \,g_x = \argmin_{i = 1, \ldots, N} \, g(x,\fni(x))$, which is generally only \emph{piecewise} smooth and can be \emph{set-valued} for arguments $x$ with multiple solutions.

\begin{remark}[Necessity of local-boundedness assumption]
    Consider the function $g = \tau \circ h$ where $\tau$ is the real-valued function $\tau:\firstvar \mapsto \frac{\firstvar^2}{2}$ and $h: \RR \times \RR \to \RR$ defined by
    $h: (\firstvar,\secondvar) \mapsto \firstvar + {\secondvar} /{\sqrt{\secondvar^2 + 1}}$.
    Note that $h(\firstvar, \cdot)$ is a diffeomorphism from $\RR$ to the open interval $(\firstvar - 1, \firstvar + 1)$, for all $x\in\RR$. Hence, $g$ is a parametric Morse function (cf. \Cref{example:morse-parametric}b). Its sets of critical points (and also local/global minima) of a fixed $x$ (given below) is not the graph of a $C^1$ function due to escape at infinity.
    \begin{equation*}
        \crit\, g_x = \argminloc\, g_x = \argmin\, g_x =
        \begin{cases}
            \left\{-\frac{\firstvar}{\sqrt{1 - \firstvar^2}}\right\} & \text{if } \firstvar \in (-1,1),\\
            \emptyset & \text{otherwise}
        \end{cases}.
    \end{equation*}
\end{remark}

A consequence of the above is that, under   \Cref{assumption:constraint-qualification}, the relaxed bilevel optimization problems \eqref{eq:crit-bilevel-optim} and \eqref{eq:loc-bilevel-optim} can be rewritten as mixed continuous-discrete variable programs in the following ways (\Cref{tab:equivalence-two-relaxations}).
\begin{table}[h]
    \centering
    \begin{tabular}{>{\centering\arraybackslash}p{0.18\textwidth}>{\centering\arraybackslash}m{0.35\textwidth}>{\centering\arraybackslash}m{0.4\textwidth}}
         \toprule
         & Critical points relaxation & Local minima relaxation \\
         \midrule
         Constrained optimization & \begin{minipage}{0.38\textwidth}
            \begin{equation*}
            \begin{aligned}
                &\underset{\firstvar \in \RR^n,\secondvar \in \RR^m}{\min} \, f(\firstvar,\secondvar)\\
				&\quad \text{s.t.} \,\, \secondvar \in \{ \fn{i}(x), 0 \leq i \leq \textcolor{black}{M} \}
            \end{aligned}
            \end{equation*}
            \vskip.05in
        \end{minipage} & 
        \begin{minipage}{0.4\textwidth}
            \begin{equation*}
            \begin{aligned}
                &\underset{\firstvar \in \RR^n,\secondvar \in \RR^m}{\min} \, f(\firstvar,\secondvar)\\
				&\quad \text{s.t.} \,\, \secondvar \in \{\fn{i}(x), 0 \leq i \leq \textcolor{black}{N} \}
            \end{aligned}
            \end{equation*}
        \end{minipage} \\
        \hline
        Mixed-integer optimization & \begin{minipage}{0.38\textwidth}
            \vskip.05in
            \begin{equation*}
            \begin{aligned}
                \underset{\firstvar \in \RR^n, 1 \leq i \leq \textcolor{black}{M}}{\min} & \, f(\firstvar,\fn{i}(x))
            \end{aligned}
            \end{equation*}
        \end{minipage} & 
        \begin{minipage}{0.4\textwidth}
            \vskip.05in
            \begin{equation*}
            \begin{aligned}
                \underset{\firstvar \in \RR^n, 1 \leq i \leq \textcolor{black}{N}}{\min} & \, f(\firstvar,\fn{i}(x))
            \end{aligned}
            \end{equation*}
        \end{minipage} \\
        \bottomrule
    \end{tabular}
    \caption{Equivalent reformulation of \eqref{eq:crit-bilevel-optim} and \eqref{eq:loc-bilevel-optim} under \Cref{assumption:constraint-qualification}. The constants $M,N$ are defined as in \Cref{prop:crit-local-structures}.}
    \label{tab:equivalence-two-relaxations}
\end{table}

\subsection{Lower level gradient descent under parametric Morse QC}
\label{sec:lower-level-gd}
\paragraph{Uniform linear convergence} The following lemma states that if $z^\star \in \argminloc \,g_{x^\star}$ (cf. \cref{lemma:splitting-function-local-minima}), then the algorithm $\algo$ behaves uniformly in terms of convergence in a neighborhood of $(x^\star,z^\star)$. These are classical results which we recall for completeness.

\begin{lemma}[Uniform linear convergence near a local minimum]
    \label{lemma:local-convexity}
	Under \Cref{assumption:lowerLevel}, if $g$ is parametric Morse, for any $z^\star \in \argminloc \,g_{x^\star}$,
	denote by $y^\star$ the implicit function associated to the equation $\grad_\secondvar g(\firstvar, \secondvar) = 0$, defined in a neighborhood of $\firstvar^\star$ and set $\varphi \colon x \mapsto f(x,y^\star(x))$. Then, there exists a neighborhood $U$ of $(\firstvar^\star,\thirdvar^\star)$  and constants $0 \leq \rho < 1$, $C > 0$, such that for all $(\firstvar, \thirdvar) \in U$, we have for all $k \in \NN$ and all $(x,z) \in U$:
    \begin{align*}
		\|\algo^k(\firstvar,\thirdvar) - y^\star(\firstvar)\| &\leq  \rho^k\|z - y^\star(x)\|,\\
		\| \nabla_z \varphi^k(x,z)\| &\leq C \rho^k, \\
		\|\nabla_x \varphi^k(x,z) - \nabla \varphi(x)\| &\leq C \rho^k,
    \end{align*}
	where $\algo$ is defined in \Cref{algo:inner_gradient_descent} and $\varphi^k$ is the function of  \eqref{eq:algorithmic-bilevel}.
\end{lemma}
\begin{proof}
	Since we have a Morse local minimizer, $\partialss^2 g(\firstvar^\star,\thirdvar^\star)$ is positive definite. Let $U = B_1 \times B_2$ be the product of two Euclidean balls centered at $\firstvar^\star$ and $\thirdvar^\star$ such that: 
	\begin{itemize}
		\item $\partialss^2 g$ remains positive definite on $U$.
		\item For any $x \in B_1$, we have $y^\star(x) \in B_2$. In other words, for all $(x,y) \in U$, $y^\star(x) \in \TL{B_2}$.
	\end{itemize}
	The first item can be satisfied using the continuity of $\partialss^2 g$ and the second one can be satisfied by reducing $B_1$ if necessary, using the local continuity of $y^\star$. We have the following properties, for all $x \in B_1$:
	\begin{itemize}
		\item The function $g_x$ is strongly convex over $B_2$.
		\item $\argmin_{B_2} g_x$ is achieved.
	\end{itemize}
    It is well known that the gradient recursion is a strict contraction with a factor $\rho < 1$ for strongly convex objectives and that its unique fixed point is the global minimizer. Without loss of generality, we may assume that the same $\rho<1$ is a uniform bound on the quantity $\|\bI - \alpha_g \partialss^2 g(\firstvar,\secondvar)\|_\mathrm{op}$ on $U$. All algorithmic iterates remain in the chosen neighborhood because of the algorithmic contraction property. This concludes the proof of the first point.

	Regarding derivatives, for a given $(x,z) \in U$, we introduce the following notations for all $k \in \NN$, and $(x,z) \in U$,
	\begin{align*}
		D_k(x,z) &= \frac{\partial}{\partial x} \algo^k(x,z)
		&C_k(x,z) = \frac{\partial}{\partial z} \algo^k(x,z)\qquad\qquad\,\\
		A_k(x,z) &= \bI - \alpha_g \partialss^2 g(x, \algo^k(x,z))
		&B_k(x,z) = - \alpha_g \partialsf^2 g(x, \algo^k(x,z))\\
		A(x) &= \bI - \alpha_g \partialss^2 g(x, y^\star(x))
		&B(x) = - \alpha_g \partialsf^2 g(x, y^\star(x))\qquad
	\end{align*}
	We have the following identities from the chain rule of differential calculus and implicit differentiation:
	\begin{align*}
		C_{k+1}(x,z) &= A_k(x,z) C_k(x,z), \\
		D_{k+1}(x,z) &= A_k(x,z) D_k(x,z) + B_k(x,z), \\
		\jac y^\star(x) &= A(x)  \jac y^\star(x) + B(x).
	\end{align*}
	We have for all $k$, uniformly in $x,z$, $\|A_k(x,z)\|_\mathrm{op} \leq \rho$. We conclude that $\|C_k(x,z)\| \leq \|C_0(x,z)\| \rho^k = m \rho^k$ which proves the second point. We also have the following:
	\begin{align}
		&D_{k+1}(x,z) - \jac y^\star(x) \nonumber\\
		=\; & A_k(x,z) D_k(x,z) + B_k(x,z) - A(x)  \jac y^\star(x) - B(x) \nonumber\\
		=\; & A_k(x,z) (D_k(x,z) -\jac y^\star(x) ) + (A_k(x,z)- A(x))  \jac y^\star(x)  + B_k(x,z)- B(x) \label{eq:mainRecursionConvergenceDerivatives}
	\end{align}
	Since $g$ is $C^3$, we have that $\partialss^2 g$ and $\partialsf^2 g$ are both Lipschitz on $U$, say with constant $\TL{M/ \alpha_g}$. By continuity, we may also consider that $H$ a uniform bound on $\|\jac y^\star(x)\|$ and set $K= M(H+1)$. Combining with the first convergence estimate, we have for all $k$:
	\begin{align*}
		\|A_k(x,z)- A(x) \| &\leq M \rho^k \|z - y^\star(x)\|, \\
		\|B_k(x,z)- B(x) \| &\leq M \rho^k \|z - y^\star(x)\|.
	\end{align*}
	Getting back to \eqref{eq:mainRecursionConvergenceDerivatives}, setting $\delta_k = \|D_k(x,z) -\jac y^\star(x)\|$ for all $k$, we have the recursion
	\begin{equation}
        \label{eq:estimation-jacobian}
	    \begin{aligned}
		      \delta_{k+1} \leq \rho \delta_k + M \rho^k\|z -y^\star(x)\| \|\jac y^\star(x)\| + M  \rho^k\|z - y^\star(x)\| \leq \rho \delta_k + K \rho^k\|z - y^\star(x)\| 
	    \end{aligned}
	\end{equation}
	with $\delta_0 \leq H \leq K$. We verify using induction that for all $k \in \NN$, $\delta_k \leq K \rho^k + k K \|z - y^\star(x)\|  \rho^{k-1}$. This is true for $k = 0$, and by induction
    \begin{equation*}
        \begin{aligned}
		      \delta_{k+1} &\leq \rho \delta_k + K \rho^k\|z - y^\star(x)\| \\
		      &\leq  K \rho^{k+1} + k K \|z - y^\star(x)\|  \rho^{k} + K \rho^k\|z - y^\star(x)\| \\
		      &=  K \rho^{k+1}  + (k+1) K \|z - y^\star(x)\|  \rho^{k},
	   \end{aligned}
    \end{equation*}
	the induction holds. Increasing $\rho$ slightly and choosing an appropriate constant $C>0$, these estimates may be upper bounded by $C\rho^k$.

	We have, for all $k \in \NN$, 
    \begin{equation}
        \label{eq:estimation-gradient}
        \begin{aligned}
		      \nabla_z \varphi^k(x,z) &= \TL{C_k}(x,z)^\top \nabla_y f(x, \algo^k(x,z))\\
		      \nabla_x \varphi^k(x,z) &= \nabla_x f(x,\algo^k(x,z)) + D_k(x,z)^\top \nabla_y f(x, \algo^k(x,z)) \\
		      \nabla \varphi(x) &= \nabla_x f(x,y^\star(x)) + \jac y^\star(x)^\top \nabla_y f(x, y^\star(x)) 
	    \end{aligned}
    \end{equation}
	and the third estimate follows, eventuallty increasing $C$, using the first estimate, since the gradient of $f$ is Lipschitz continuous and bounded on $U$ since $f$ is $C^2$.
\end{proof}

\paragraph{Convergence to local minima with high probability over initialization}

Another interesting property of parametric Morse functions is that with high probability over initialization $(\firstvar,\thirdvar)$, \Cref{algo:inner_gradient_descent} eventually finds a local minimum of $g_x$ for large $k$. This is of course in the line of classical results for Morse functions \cite{pemantle1990nonconvergence,goudou2009gradient} and based on the stable-unstable manifold theorems.

\begin{proposition}[Uniform convergence to local minima over compact initialization]
    \label{theorem:uniform-convergence}
    Let $g$ satisfy all the assumptions in \Cref{assumption:lowerLevel} -- except semi-algebraicity -- and assume in addition that $g$ is parametric Morse \footnote{Note that we do not assume local boundedness of $\crit g_x$ as in \Cref{assumption:constraint-qualification}.} (cf. \Cref{def:parametric_morse}).
    For any compact set $C \subseteq \RR^n \times \RR^m$, any $\delta, \Delta > 0$, there exists a set $A \subseteq \RR^n \times \RR^m$ (only dependant on $\Delta$) and a natural number $K \in \NN$ (dependant on $C, \delta$ and $\Delta$) such that:
    \begin{enumerate}
        \item[i)] The Lebesgue measure of $A$ is smaller than $\Delta$, i.e., $\mu(A) \leq \Delta$.
        \item[ii)] We have: 
        $$\forall (\firstvar,\thirdvar) \in C \setminus A, \forall k \geq K \quad \dist(\cA^k(\firstvar, \thirdvar), \argminloc g_x) \leq \delta.$$ 
        where $\dist(x, A):= \inf_{y \in A} \|x-y\|_2$.
    \end{enumerate}
\end{proposition}

\begin{proof}
    We use two following facts:
    \begin{enumerate}[leftmargin=*]
        \item Since $g_x$ is Morse ($g$ is parametric Morse), $g_x$ satisfies the KL inequality \cite[Section 4]{attouch2010proximal}. By assumptions, $g$ is $L_g$-smooth, coercive, and the learning rate $0 < \alpha_g < 1 / L_g$. Using \cite[Theorem 3.2]{attouch2013Convergence}, we can prove that for any $(\firstvar,\thirdvar) \in \RR^m \times \RR^n$, the sequence $\cA^k(\firstvar, \thirdvar)$ will converge to a critical point of $g_x$.
        \item Due to the result of \cite{pemantle1990nonconvergence,goudou2009gradient,lee2016gradient,panageas2017gradient}, for a fixed $\firstvar$, the set of $\thirdvar$ such that the sequence $\cA^k(\firstvar, \thirdvar)$ converges to a strict saddle point (i.e. a critical point $y$ of $g_x$ such that $\nabla_{yy}^2 g(x,y)$ has at least one negative eigenvalue) is of Lebesgue measure zero. Using the Fubini-Tonelli theorem, we have the set 
        $$\cB:= \{(\firstvar, \thirdvar) \mid \{\cA^k(\firstvar,\thirdvar)\}_{k \in \NN} \text{ converges to a saddle point of } g_x\},$$
        is of zero Lebesgue measure. 
        
    \end{enumerate}
    Consider $\bar{\cB}$ the complement of the set $\cB$. By combining the first two remarks, we can conclude that $\bar{\cB}$ contains $(\firstvar,\thirdvar)$ whose $\{\cA^k(\firstvar,\thirdvar)\}_{k \in \NN}$ converges to a local minimum $\thirdvar^\star$ of $g_x$.
    
    To finish the proof, for any $\Delta > 0$, we take an open set $A$ such that $\mu(A) \leq \Delta$ such that $\cB \subseteq A$. Note that this is possible because $\mu(\cB) = 0$. Since $A$ is open, $C \setminus A$ is compact.

    Consider a point $(\firstvar, \thirdvar) \in C \setminus A$. Let $\thirdvar^\star$ be the local minimizer of $g_x$ that $\cA^k(\firstvar,\thirdvar), k \in \NN$ converges to. By \Cref{lemma:local-convexity}, there exists an open neighborhood $\TL{U_1(\firstvar, \thirdvar^\star)}$ of $(\firstvar, \thirdvar^\star)$ such that $\forall (\firstvar', \thirdvar') \in U_1(x,z^\star), \forall k \in \NN$, $\|\cA^k({\firstvar'},\thirdvar') - \bar{\secondvar}(\firstvar')\| \leq \|\thirdvar' - \bar{y}(x')\| \rho^k$ for some $0 < \rho < 1$ and the function $\bar{\secondvar}$ is defined as in \Cref{lemma:local-convexity}. Note that $\bar{y}(x')$ is also the local minimizer of $g(x',\cdot)$ (using the same argument as in the proof of \Cref{lemma:splitting-function-local-minima}).
    
    Let $k(\firstvar, \thirdvar)$ be the smallest number such that: $(\firstvar, \cA^{k(\firstvar,\thirdvar)}(\firstvar,\thirdvar)) \in U_1(x,z^\star) \cap B((x,z^\star), \delta)$ (which exists due to the convergence assumption). Since $\cA^{k(\firstvar,\thirdvar)}$ is a continuous function in both variables $\firstvar$ and $\thirdvar$, there exists an open neighborhood $U_2(\firstvar, \thirdvar)$ such that its image under $\cA^{k(\firstvar, \thirdvar)}$ is included in $U_1(x, z^\star)$. By construction, we have in addition:  
    \begin{equation*}
        \|\bar{y}(x') - \cA^k(x',z')\| \leq \|\bar{y}(x') - \cA^{k(x,z)}(x',z')\| \leq \delta, \forall k \geq k(x,z), \forall (x',z') \in \TL{U_2(x,z)}.
    \end{equation*}
    Note that $\cU := \{U_2(\firstvar, \thirdvar) \mid (\firstvar, \thirdvar) \in C \setminus A\}$ is an open coverage of the compact set $C \setminus A$. Therefore, there exists a finite subset $D \subseteq C \setminus A$ of such that:
    \begin{equation*}
        C \setminus A = \bigcup_{(\firstvar, \thirdvar) \in D} U(\firstvar, \thirdvar).
    \end{equation*}
    This implies the result immediately because one can take $K = \max_{(\firstvar, \thirdvar) \in D} k(\firstvar, \thirdvar)$.
\end{proof}
\begin{remark}[On the necessity of compactness] Let us observe that the compactness of $C$ is necessary: a trivial counterexample is $g(x,y) = (x-y)^2$ and $C = \RR^{n + m}$. In this setting, there does not exist any $A, K$ as in \Cref{theorem:uniform-convergence} (and we leave the verification of this claim for the readers)
\end{remark}

\paragraph{Attractivity of the connected components of $\graph ( \Argminlocg )$} We shall also need a modified version of the attraction properties of $\graph ( \Argminlocg )$ on a compact set. 
\begin{lemma}[Attraction properties of the lower level constraints]
    \label{lemma:universal-contraction}
    Consider a compact set $\cS \subseteq \RR^n$. Under \Cref{assumption:constraint-qualification} and \Cref{assumption:lowerLevel}, there exist constants $r:=r(\cS) > 0, \rho:= \rho(\cS) \in (0,1), c:=c(\cS) > 0$ such that for all $(\firstvar, \thirdvar) \in \cS \times \RR^m$:
    \begin{enumerate}[leftmargin=*]
        \item If $\dist((\firstvar, \thirdvar), \TL{\graph} \Argminlocg) \leq r$, there is a \emph{unique} $C^2$ function $y$ with $\TL{\graph} y \subseteq \Argminlocg$  such that: 
        $$\|\cA^k(\firstvar, \thirdvar) - y(\firstvar)\| \leq \rho^k\|\thirdvar - y(\firstvar)\|.$$
        \item If $\dist((\firstvar, \cA^{l}(\firstvar,\thirdvar)), \TL{\graph} \Argminlocg) \leq r$, there is a \emph{unique} $C^2$ function $y$ with  $\graph y \subseteq \Argminlocg$ such that:
        \begin{equation*}
            \|\nabla_\firstvar \varphi^{k+l}(\firstvar,\thirdvar) - \nabla_\firstvar \varphi^\infty(\firstvar,\thirdvar)\|_2 \leq (\|\jacnospace_x \cA^l(x,z) - \jacnospace_x y(x)\| + k)c\rho^k
        \end{equation*}
        where $\varphi^k(\firstvar,\thirdvar) := f(\firstvar, \cA^k(\firstvar,\thirdvar))$ and $\varphi^\infty(\firstvar,\thirdvar) := f(\firstvar, y(\firstvar))$.
    \end{enumerate}
\end{lemma}

\begin{proof}
    Using \Cref{assumption:constraint-qualification} and \Cref{assumption:lowerLevel}), thanks to \Cref{prop:crit-local-structures}, there exists $N$ functions $\fn{1}, \ldots, \fn{N}$ which are $C^2$ and have distinct values such that $\Argminlocg$ is partitioned into their graphs, i.e., \cref{lemma:splitting-function-local-minima}. Therefore, the condition:
    $$\dist((\firstvar, \thirdvar), \TL{\graph} \Argminlocg) \leq r$$
    is equivalent to:
    $$\min_{i = 1, \ldots, N} \dist((\firstvar, \thirdvar), \graph \fn{i}) \leq r.$$
    Using this observation, we construct $(r, \rho, c)$ as follows: For any $(\firstvar, \thirdvar) \in \Argminlocg \cap \cS \times \RR^m$, let $y = \fn{i}$ (for some $i \in \{1, \ldots, N\}$) be the implicit function given by $\nabla_y g(x,y) = 0$ in a neighbor of $(x,z)$. By \Cref{lemma:local-convexity}, there exists an open ball centered at $(\firstvar, \thirdvar)$ with radius $r(\firstvar,\thirdvar)$ such that for any point $(\bar{\firstvar}, \bar{\thirdvar})$ in this open ball, there are two constants $K(\firstvar,\thirdvar) > 0, \rho(\firstvar,\thirdvar) \in (0,1)$ depending on $(\firstvar, \thirdvar)$ satisfying:
    \begin{enumerate}[leftmargin=*]
        \item Using the bound in \Cref{lemma:local-convexity}, we immediately get:
        \begin{equation}
            \label{eq:bound-iterates}
            \|\cA^k(\firstvar, \thirdvar) - y(\firstvar)\| \leq \rho(\firstvar,\thirdvar)^k\|\thirdvar - y(\firstvar)\|.
        \end{equation}
        \item By applying \eqref{eq:estimation-jacobian} $k$ times, we obtain:
        \begin{equation}
            \label{eq:bound-jacobian}
            \|\jacnospace_x \cA^{k + l}(x,z) - \jacnospace_x y(x)\| \leq (\|\jacnospace_x \cA^{l}(x,z) - \jacnospace_x y(x)\| + K(x,z)\TL{k}) \rho(x,z)^k    
        \end{equation}
         Combining \eqref{eq:estimation-gradient}, \eqref{eq:bound-iterates}, \eqref{eq:bound-jacobian} and the Lipschitz continuity of $\nabla_x f, \nabla_y f$ in the ball centered at $(x,z)$ with radius $r(x,z)$, we conclude that there is a constant $c(x,z)$ such that:
         \begin{equation*}
            \|\nabla_\firstvar \varphi^{k+l}(\firstvar,\thirdvar) - \nabla_\firstvar \varphi(\firstvar)\|_2 \leq (\|\jacnospace_x \TL{\cA^l(x,z)} - \jacnospace_x y(x)\| + K(x,z)\TL{k})c(x,z)\rho(x,z)^k.
        \end{equation*}
    \end{enumerate}

    By \Cref{lemma:splitting-function-local-minima}, $ [\graph\,  \Argminlocg]  \cap \left( \cS \times \RR^m\right)$ is the union of $N$ graphs of $C^2$ functions $\fn{1}, \ldots, \fn{N}$, restricted to $\cS$. By compactness, there exists a finite set: $
    {\cal F} \subseteq  [\graph\,  \Argminlocg]  \cap  \left( \cS \times \RR^m\right)$ and a finite covering:
    \begin{equation*}
      [\graph\,  \Argminlocg] \cap \left( \cS \times \RR^m\right) \subseteq \bigcup_{(\firstvar, \thirdvar) \in  {\cal F} } B((\firstvar, \thirdvar), r{(\firstvar, \thirdvar)}).
    \end{equation*}
    The lemma can be proved by choosing
    \begin{equation*}
            r := \min_{(\firstvar,\thirdvar) \in  {\cal F} } \frac{r{(\firstvar, \thirdvar)}}{2} > 0, \quad
            c := \max_{(\firstvar,\thirdvar) \in  {\cal F} } c(\firstvar, \thirdvar) > 0, \quad
            \rho := \max_{(\firstvar,\thirdvar) \in  {\cal F} } \rho(\firstvar, \thirdvar) \in (0,1).
    \end{equation*}
    It remains to show the uniqueness of the function $y$. Note that with this choice of $r$, from each $(\firstvar,\thirdvar)$ whose distance from $\Argminlocg$ is at most $r(\cS)$, $y$ is the only $C^2$ function whose $y(\firstvar) = \lim_k \cA^k(\firstvar,\thirdvar)$.
\end{proof}

\section{Convergence of the single-step multi-step gradient method}
\label{sec:single-step-multi-step}

Under \Cref{assumption:lowerLevel},  \Cref{assumption:constraint-qualification}), and natural assumptions on $f$, \Cref{algo:altenate-method} is actually an {\em inexact gradient descent method} \cite{tadic2011asymptotic,bolte2024inexact,arbel:hal-03869097} with non-vanishing steps on a function $f(x, y(x))$ where $y(x)$ is a $C^2$ function giving a local minimizers of $g_x$.

Formally, this can be stated as follows. Let $\varphi \colon x\mapsto f(x,y(x))$ be the value function. For all $\epsilon>0$, there is a lower bound on the number of lower level gradient steps $k$ in \Cref{algo:inner_gradient_descent}, such that \Cref{algo:altenate-method} can be rewritten as follows
\begin{equation*}
        \firstvar_{\ell+1} = \firstvar_{\ell} - \alpha_f (\grad_\firstvar \varphi(x) + \xi_\ell) \quad \text{where} \quad \|\xi_\ell\| \leq \epsilon \mbox{ for all }\ell \in \NN.
\end{equation*}

\subsection{Main convergence result} 

Before stating our main result let us recall a definition.
\begin{definition}[$\epsilon$-critical points and critical values]
    Given a differentiable function $h:\RR^n\to\RR$, the sets of \emph{$\epsilon$-critical points $\crit_\epsilon h$} and \emph{$\epsilon$-critical values $\vcrit_\epsilon h$} of $h$ are respectively defined defined through:
    \begin{equation*}
        \begin{aligned}
            \crit_\epsilon h &:= \{x \mid \|\nabla h(x)\| \leq \epsilon\},\\
            \vcrit_\epsilon h &:= h(\crit_\epsilon h).\\
         \end{aligned}
    \end{equation*}
\end{definition}

\smallskip
In the following, we consider a random choice of initialization whose distribution is absolutely continuous w.r.t. Lebesgue measure. It covers all practically used random initializations such as those based on Gaussian or uniform distributions, see for example \cite{kaiming2015delving}.

\begin{theorem}[Convergence property of \Cref{algo:altenate-method}]
    \label{theorem:convergence-property-alternate-method}
    Suppose that \Cref{assumption:lowerLevel} and \Cref{assumption:constraint-qualification} hold 
    and assume in addition that:
    \begin{itemize}[nosep]
    \item there exists some $\lambda > 0$ such that $\inf_{x,y} f(x,y) - \lambda\|x\|^2 > -\infty$,
	\item the initialization $x_0$ is chosen randomly with a distribution absolutely continuous with respect to the Lebesgue measure. 
    \end{itemize}
	Then, for any $\epsilon, \Delta> 0$, there is a sufficiently large inner iterations $k$ (cf. \Cref{algo:inner_gradient_descent}) and constants $C>0$ and $\rho \in (0,1)$, such that with probability at least $1 - \Delta$, there exists $y = \fn{i}, i \in \{1, \ldots, N\}$ (cf. \Cref{lemma:splitting-function-local-minima}) satisfying:
    \begin{equation*}
        \begin{aligned}
           & \lim_{\alpha_f \to 0^+} \limsup_{\ell \to \infty} &\;\; \dist(\varphi(x_\ell(\alpha_f)), \vcrit_\epsilon \varphi) = 0\\
           & \limsup_{\alpha_f \to 0^+} \limsup_{\ell \to \infty} &\;\; \dist(x_\ell(\alpha_f), \crit \varphi) \leq C\epsilon^\rho.\\
        \end{aligned}
    \end{equation*}
    where $\varphi(x) = f(x,y(x))$ is the value function, and $x_\ell(\alpha_f)$ is the sequence generated by \Cref{algo:altenate-method} with step size $\alpha_f$.
\end{theorem}

\begin{remark}
    \label{remark:coercivity}
    In \Cref{theorem:convergence-property-alternate-method}, the assumption that $f$ is semi-algebraic (cf. \Cref{assumption:lowerLevel}) and $\inf_{x,y} f(x,y) - \lambda\|x\|^2 > -\infty$ is equivalent to $f(x,y) = h(x,y) + \lambda \|x\|^2$ with $h(x,y)$ is semi-algebraic and bounded from below. This is a coercivity assumption. This condition is satisfied in many practical settings (e.g., $f$ is a regularized version of a semi-algebraic non-negative data fitting term or a distance function $h(x,y)$). In addition, given a semi-algebraic function $f$ where $\inf f > -\infty$, by adding a $\|\cdot\|^2$ perturbation, the assumption in \Cref{theorem:convergence-property-alternate-method} on $f$ will be satisfied. 

    \TL{Except \Cref{theorem:uniform-convergence} and \Cref{lemma:local-convexity}, semi-algebraicity is necessary in the announcement of \Cref{theorem:convergence-property-alternate-method} and intermediate results used in its proof. They ensure that the minimum values of certain finite quantities are strictly bigger than zero and that the inexact gradient method approximates the set of critical points \cite{bolte2024inexact}. Relaxing this condition is an interesting direction for future works, given our remark \Cref{remark:uniform-boundedness} on the uniform boundedness of the cardinality of $\crit g_x$ and \cite[Remark 3.2]{bolte2024inexact}.}
\end{remark}
Before proving \Cref{theorem:convergence-property-alternate-method}, we comment on related results in the literature, notably \cite[Proposition 3]{arbel:hal-03869097} which studies the same algorithm:
\begin{enumerate}
    \item Our proof of \Cref{theorem:convergence-property-alternate-method} is based on \Cref{lemma:first-iter-dynamics} and \Cref{lemma:remaining-iterations}, which imply that $y_\ell$ stays close to a local minimum of $g_{x_\ell}$, for all $\ell \geq 1$, with high probability over initialization. It allows us to relax the assumption used in \cite[Proposition 1]{arbel:hal-03869097}: $(x_\ell,y_\ell)$ converges to $(x^\star, y^\star)$ where $y^\star \in \argminloc \, g(x^\star,\cdot)$.
    \item The fact that \Cref{algo:altenate-method} is equivalent to a biased gradient descent on a certain solution function was previously observed, see \cite{ji2021bilevel} (strongly convex $g_x$) and \cite{arbel:hal-03869097} (local analysis for non-convex $g_x$). Our improvement on this aspect follows from the coercivity assumption on $f$ (cf. \Cref{remark:coercivity}). This is sufficient for the sequence to be bounded and allows for a global analysis of \Cref{algo:altenate-method} for non-convex $g_x$ without a priori assumption on the sequence. In addition, by harnessing the semi-algebraicity, we obtain estimates for the distance between the sequence $\{x_\ell\}_{\ell \in \NN}$ and the set of critical points of the bilevel problem, even for finite (and sufficiently large) $k$, whereas existing results \cite{arbel:hal-03869097,liu2021towards} have to settle with the limit $k \to \infty$ {(remind that $k$ is the number of inner iterations, and not to be confused with the number of outer iterations $\ell$)}.
\end{enumerate}
\TL{In conclusion, we think it is possible to relax the parametric Morse assumption to the parametric Morse-Bott one in \Cref{theorem:convergence-property-alternate-method} by combining analysis from this article and \cite{arbel:hal-03869097}. While this relaxation might be non-trivial and mathematically interesting, its significance is likely to be minor since the parametric Morse-Bott property is not generic, as that is the case for the parametric Morse property (see \Cref{ex:morseParamNongeneric}).}

\subsection{Gradient descent on the lower level}

We first prove that \Cref{algo:inner_gradient_descent} reaches a neighborhood of local minimizer of $g_{x_0}$ for a typical initialization. The higher the number of iterations $k$, the smaller the neighborhood and the error term $\|\xi_0\|$.
\begin{lemma}[First iterations approximate local minima]
    \label{lemma:first-iter-dynamics}
    Suppose that \Cref{assumption:lowerLevel} and \Cref{assumption:constraint-qualification} holds. 
    
	For any compact $C \subset \RR^n \times \RR^m$, any $\Delta, r, \epsilon > 0$, there exists $A \subset \RR^n\times \RR^m$ with Lebesgue measure at most $\Delta$, an integer $K \in \NN$ and ${\alpha} > 0$ such that if one runs \Cref{algo:altenate-method} with $k \geq K$, $ 0<\alpha_f \leq {\alpha}$ and $(x_0,y_0) \in C \setminus A$, there exists a $C^2$ function $y$ such that $\graph y \subseteq \graph\, \Argminlocg$ (see \Cref{prop:crit-local-structures}) and that:
    \begin{equation*}
        \begin{aligned}
            %\|\cA^k(x_0,z_0) - y(x_0)\| & \leq r/2, \\
			\|\nabla_x \varphi^k(x_0,y_0) - \nabla_x \varphi(x_0) \| & \leq \epsilon,\\
            \dist((x_1,y_1), \graph y) & \leq r,
        \end{aligned}
    \end{equation*}
	where $\varphi(x) = f(x, y(x))$, $\varphi^k(x,z) = f(x, \cA^k(x,z))$, {$(x_0, y_0), (x_1,y_1)$ are the zeroth (initialization) and the first iterates of \Cref{algo:altenate-method}, respectively}.
\end{lemma}

\begin{proof} We use several results and the notation of \Cref{lemma:universal-contraction}.
	By \Cref{theorem:uniform-convergence}, we choose $\bar{K} := K(\cS, r'/2, \Delta)$ where: 
    \begin{enumerate}
		\item $\cS = \Pi(C)$ where $\Pi(\cdot)$ is the projection onto the first $n$ coordinates.
        \item $r' = \min(r(\cS), r)$, $r(\cdot)$ is defined as in \Cref{lemma:universal-contraction}).
        \item $A \subseteq \RR^n \times \RR^m$, has Lebesgue measure at most $\mu(A) \leq \Delta$ as given by \Cref{theorem:uniform-convergence}. 

    \end{enumerate}
    Then for any $(x_0,y_0) \in C \setminus A$ we have:
    $$\dist(\cA^{\bar{K}}(\firstvar_0, \secondvar_0), \argminloc\,g_{\firstvar_0}) \leq r'/2.$$
    Therefore, there exists a $C^2$ function $y = \fn{i}, i \in \{1, \ldots, N\}$ (where $N$ and $\fn{i}$ are defined as in \Cref{lemma:splitting-function-local-minima}) such that:
    \begin{equation*}
        \|\cA^{\bar{K}}(\firstvar_0, \secondvar_0) - y(\firstvar_0)\| \leq r'/2 \leq r / 2.
    \end{equation*}

    Let $\varphi(x) = f(x, y(x))$.
    Our next step is to bound $\|\nabla_\firstvar \varphi^k(x_0,y_0) - \nabla_\firstvar \varphi({\firstvar}_0)\|_2$ (note that we cannot use \Cref{lemma:universal-contraction} yet because we need $\|\secondvar_0 - y(\firstvar_0)\| \leq r(\Pi(\cS))$ to do so). To this end, we only consider iterations where $\cA^k(\firstvar_0,\secondvar_0)$ already entered the band of radius $r(\cS)$ around the graph of $y$. Let $\rho = \rho(\cS) \in (0,1), c = c(\cS) > 0$ as in \Cref{lemma:universal-contraction}. For $k \geq \bar{K}$, we have:
    \begin{equation*}
        \|\cA^k(\firstvar_0, \secondvar_0) - y(\firstvar_0)\| \leq r/2\rho^{k-\bar{K}} \leq r/2,
    \end{equation*}
    by a simple induction. Moreover, also by \Cref{lemma:universal-contraction},
    \begin{equation*}
        \|\nabla_\firstvar \varphi^{k}(\firstvar_0,\secondvar_0) - \nabla_\firstvar \varphi(\firstvar_0)\| \leq (\|\jacnospace_x \cA^{\bar{K}}(x_0,y_0) - \jacnospace_x y(x_0)\| + k - \bar{K})c\rho^{k-\bar{K}},
    \end{equation*}
    by our assumption on $k$. Increasing $k$ if necessary, the right hand side can be made smaller than $\epsilon$, because $\jacnospace_x \cA^{\bar{K}}(x,z)$ and $\jacnospace_x y(x)$ are continuous, and thus, bounded on the bounded set $\cS$. 

    Finally, define:
    \begin{equation}
        \label{eq:choice-alpha}
		{\alpha} = \min_{i = 1, \ldots, N, x \in \Pi(\cS)} \frac{r}{2(\|\nabla_x  f \circ (\mathtt{id}, \fn{i})(x) \| + \epsilon)},
    \end{equation}
    we have:
    \begin{equation}
        \label{eq:update-new-algo}
        \begin{aligned}
            \dist((\firstvar_1, \secondvar_1), \graph y) &\leq \|(\firstvar_1, \secondvar_1) - (\firstvar_0, y(\firstvar_0))\|\\
            & \leq \|\firstvar_1 - \firstvar_0\| + \|\secondvar_1 - y(\firstvar_0)\|\\
            & \leq {\alpha} \|\nabla_\firstvar \varphi^k(\firstvar_0, \secondvar_0)\| + \|\cA^k(\firstvar_0,\secondvar_0) - y(\firstvar_0)\|\\
			& \leq \frac{r}{2(\|\nabla_x  \left(f \circ (\mathtt{id}, y)\right)(x) \| + \epsilon)}\|\nabla_\firstvar \varphi^k(\firstvar_0, \TL{\secondvar_0})\| + \|\cA^k(\firstvar_0,\secondvar_0) - y(\firstvar_0)\|\\
            & \leq r/2 + r/2 = r,
        \end{aligned}
    \end{equation}
    by the choice of ${\alpha}$ in \eqref{eq:choice-alpha}.
\end{proof}

\subsection{Convergence proof of \Cref{algo:altenate-method}}

We establish, in the spirit of ``path following methods", see e.g. \cite{allgower1997numerical} and references therein, that the dynamics evolves  near the manifold of local minimizers corresponding to the local minimum selected in \Cref{lemma:first-iter-dynamics}.
\begin{lemma}[Dynamics of remaining iterations]
    \label{lemma:remaining-iterations}
    Assume that \Cref{assumption:lowerLevel} and \Cref{assumption:constraint-qualification} hold, then for all $M > 0, \epsilon > 0$, there exists $r > 0, K \in \NN, {\alpha} > 0$ such that: for any $\fn{\star} = \fn{i}, i \in \{1, \ldots, N\}$ ($\fn{i}$ and $N$ are defined as in \Cref{lemma:splitting-function-local-minima}), for all $\|x\| \leq M, k \geq K, \alpha_f \leq \alpha$, if $\dist((x,z), \graph \fn{\star}) \leq r$, then:
    \begin{equation}
        \begin{aligned}
			\|\nabla_x \varphi(x) - \nabla_x \varphi^k(x,z)\| &\leq \epsilon,\\
            \dist((x',z'), \graph \fn{\star})&\leq r\\
        \end{aligned}
    \end{equation}
	where $\varphi(x) = f(x,y^*(x))$, and $(x',y')$ is the iterate update of $(x,z)$ performed by \Cref{algo:altenate-method}, {i.e., if $(x,z) = (x_k, y_k)$, then $(x',z') = (x_{k+1},y_{k+1})$}.
\end{lemma}
\begin{proof}
    Let $\cS = B(0, M) \subseteq \RR^n$ and consider $r = r(\cS), \rho = \rho(\cS)$ and $c = c(\cS)$ as in \Cref{lemma:universal-contraction}. Using \Cref{lemma:universal-contraction}, we get:
    \begin{equation*}
        \begin{aligned}
            \|\cA^k(\firstvar, \thirdvar) - \fn{\star}(\firstvar)\| &\leq \rho^k\|\thirdvar - \fn{\star}(\firstvar)\|,\\
            \|\nabla_\firstvar \varphi^{k}(\firstvar,\thirdvar) - \nabla_\firstvar \varphi(\firstvar)\|_2 &\leq (\|\jacnospace_x \fn{\star}(x)\| + k)c\rho^k.
        \end{aligned}
    \end{equation*}
    where $\varphi(x) = f(x,\fn{\star}(x))$. Thus, for $k$ sufficiently large, we get:
    \begin{equation*}
        \begin{aligned}
            \|\cA^k(\firstvar, \thirdvar) - \fn{\star}(\firstvar)\| &\leq r/2,\\
            \|\nabla_\firstvar \varphi^{k}(\firstvar,\thirdvar) - \nabla_\firstvar \varphi(\firstvar)\|_2 &\leq \epsilon.
        \end{aligned}
    \end{equation*}
    Finally, to show that $\dist(x',y') \leq r$, we choose ${\alpha}$ as in \eqref{eq:choice-alpha} and repeat the same argument as in the proof of \Cref{lemma:first-iter-dynamics}.
\end{proof}

In the light of \Cref{lemma:first-iter-dynamics}, \Cref{lemma:remaining-iterations}, the convergence of \Cref{algo:altenate-method} boils down to the study of an \emph{inexact gradient descent} methods provided that the iterates remain \emph{bounded} (otherwise, \Cref{lemma:remaining-iterations} would fail). 
To establish the convergence results for \Cref{algo:altenate-method}, we first introduce a condition on $f,g$ ensuring that iterates are uniformly bounded. This condition is also used in \cite{bolte2024inexact} to study inexact gradient descent methods. We start with a descent mechanism, the arguments are classical.

\begin{lemma}[Quasi-descent lemma]
    \label{lemma:quasi-descent}
    Consider a $C^2$, semi-algebraic, and lower-bounded function $f$ whose $\crit_\epsilon f$ is non-empty and bounded for some $\epsilon > 0$. For any $l > \max_{z \in \crit_\epsilon f} f(z)$, there exist ${\alpha} > 0$ and $M>0$, such that for any $\alpha_f$,  $0<\alpha_f \leq {\alpha}$, and any $z \in \RR^n$ such that $f(z) \leq l$, the sequence defined by:
    \begin{equation}
        \label{eq:update-inexact-gd}
        x_0 = z, \qquad \qquad x_{\ell+1} = x_\ell - \alpha_f (\nabla f(x_\ell) + \xi_\ell), \|\xi_\ell\| \leq \epsilon, \forall \ell \in \NN,
    \end{equation}
	satisfies $f(x_\ell) \leq l$ and $\|x_\ell\| \leq M$, for all $\ell \in \NN$.
\end{lemma}
\begin{proof}
	\cite[Lemma 3.4]{bolte2024inexact} ensures that $f$ is coercive. Set $S = \{x, \, f(x) \leq l\}$ which is compact and set $B_{0}$ a ball of radius $M_0$ containing $S$. Denote by $L_0$ a Lipschitz constants of $\nabla f$ on $B_{0}$ ({$L_0$} exists because $f$ is $C^2$). Set $M = M_0 + \frac{1}{L_0} (\max_{\|z \| \leq M_0} \|\nabla f(z)\| + \epsilon)$, denote by $B$ the corresponding ball (with the same center as $B_0$), set $L_f$ and $L_{\nabla f}\geq L_0$ the Lipschitz constants of $f$ and $\nabla f$ on $B$ respectively. Set
	\begin{align*}
		\alpha & = \min \left\{ \frac{1}{L_{\nabla f}}, \frac{l - \max_{z \in \crit_\epsilon f} f(z)}{L_f(L_f + \epsilon)}  \right\} 
	\end{align*}
	Fix arbitrary $x \in S$, $\alpha_f \in [0,\alpha]$ and $\|\xi\| \leq \epsilon$, let us show that $y = x - \alpha_f (\nabla f(x) + \xi) \in S$. We have $x \in B_0 \subset B$ and $\alpha_f \|\nabla f(x) + \xi\| \leq \frac{1}{L_0} (\max_{\|z \| \leq M_0} \|\nabla f(z)\| + \epsilon)$ so that $y = x - \alpha_f (\nabla f(x) + \xi) \in B$. 
	First if $x \not \in \crit_\epsilon f$, we can use the descent lemma, since $x,y \in B$, 
	\begin{align*}
		f(y) -  f(x)  &\leq \left\langle \nabla f(x), y-x \right\rangle + \frac{L_{\nabla f}}{2} \|y -x\|^2 \leq \left\langle \nabla f(x), y-x \right\rangle + \frac{1}{2 \alpha_f} \|y -x\|^2 \\ 
        & = \frac{1}{2 \alpha_f} \left( 2\left \langle \alpha_f \nabla f(x), y-x \right\rangle +  \|y-x\|^2 \right) \\
        &= \frac{1}{2 \alpha_f} \left(\left\|y-x + \alpha_f \nabla f(x)\right\|^2 - \left\|\alpha_f \nabla f(x)\right\|^2 \right) \\
		&= \frac{\alpha_f}{2} \left(\|\xi\|^2 - \|\nabla f(x)\|^2\right) < 0.
	\end{align*}
	We deduce that $y \in S$. Second, if $x \in \crit_\epsilon$, we have $\|y-x\| = \alpha_f\| \nabla f(x) + \xi\| \leq \alpha_f(L_f + \epsilon)$, and
	\begin{align*}
		f(y) \leq f(x) + L_f \|y-x\| \leq f(x) + \alpha_f L_f(L_f + \epsilon) \leq f(x) + l - \max_{z \in \crit_\epsilon f} f(z) \leq l
	\end{align*}
	and $y \in S$. This proves the desired result. 
\end{proof}

Combining \Cref{lemma:first-iter-dynamics,lemma:remaining-iterations,lemma:quasi-descent}, we get the following result:
\begin{theorem}[Inexact gradient descent equivalence of \Cref{algo:altenate-method}]
    \label{theorem:equivalence-inexact-gd}
	Under \Cref{assumption:lowerLevel} and \Cref{assumption:constraint-qualification}, assume that $\inf_{x,y} f(x,y) - \lambda\|x\|^2 > -\infty$ for some $\lambda > 0$.
	For any compact set $C \subseteq \RR^n \times \RR^m$, any $\epsilon, \Delta > 0$, there exists a set $A \subseteq \RR^n \times \RR^m$ with Lebesgue measure at most $\Delta$, $K \in \NN, {\alpha} > 0, M > 0$ such that if $x_0 \in C \setminus A$, $k \geq K, \alpha_f \leq \alpha$:
    \begin{enumerate}[leftmargin=*]
        \item The sequence $\{(x_\ell,y_\ell)|\ell \in \NN\}$ produced by \Cref{algo:altenate-method} is bounded in $B(0,M) \subseteq \RR^n \times \RR^m$.
		\item There is $y = \fn{i}, i \in \{1, \ldots, N\}$ (c.f. \Cref{prop:crit-local-structures}) such that, $\|\nabla_x \varphi(x_\ell) - \nabla_x \varphi^k(x_\ell,y_\ell)\| \leq \epsilon, \forall \ell \in \NN$, where $\varphi(x) = f(x,y(x))$.
    \end{enumerate}
\end{theorem}
\begin{proof}
    Since $f$ and $\fn{i}, i = 1, \ldots, N$ are semi-algebraic (cf. \Cref{prop:crit-local-structures}), $f_i \colon x \mapsto f(x,\fn{i}(x))$ is semi-algebraic. Moreover, each function $f_i - \lambda\|\cdot\|^2$, $i = 1, \ldots, N$ is trivially lower-bounded and each set $\crit_{\epsilon} f_i$, $i = 1, \ldots, N$ is bounded as proved in \cite[\TL{Proposition} 3.1]{bolte2024inexact}. The critical values $\vcrit_{\epsilon} f_i$ are also bounded for $i = 1,\ldots, N$ by continuity.

	Choose $l>0$ such that
	\begin{align*}
		&l > \max\{\vcrit_{\epsilon}\,  f \circ (\mathtt{id}, \fn{i})\}& l \geq \max_{x,y \in C} \left\{f(x, \fn{i}(x))\right\},\, \forall i = 1, \ldots, N.
	\end{align*}
	We choose $M > 0, \alpha_1 > 0$ as the maximal values given by \Cref{lemma:quasi-descent} applied to each $f_i$, $i = 1, \ldots, N$. Note that it follows from the statement of \Cref{lemma:quasi-descent} that the ball of radius $M$ should contain the sublevel sets of value $l$ for each $f_i$, $i=1 \ldots,N$.
	
	With $\epsilon$ and $M$ thus given, we choose $(r, K_1, \alpha_2)$ equal to the quantities $(r, K, \alpha)$ of \Cref{lemma:remaining-iterations}. In particular, if $\|x_\ell\| \leq M$, $k \geq K_1, \alpha_f \leq \alpha_2$ and $\dist((x_{\ell}, y_{\ell}), \graph \fn{i}) \leq r$, we have:
    \begin{enumerate}
        \item $x_{\ell+1}$ is updated by \eqref{eq:update-inexact-gd} applied to $f_i \colon x \mapsto f(x,\fn{i}(x))$.
        \item $\dist((x_{\ell+1}, y_{\ell+1}), \graph \fn{i}) \leq r$.
    \end{enumerate}
    for any index $i \in \{1, \ldots, N\}$.

	Next, given $C$, $\Delta > 0, r > 0, \epsilon > 0$, we obtain a subset $A$ of Lebesgue measure at most $\Delta$, $K_2 \in \NN, \alpha_3 > 0$ equal to $K$ and $\alpha$ defined as in \Cref{lemma:first-iter-dynamics}, respectively and a $C^2$ solution $y$ such that for any $(x_0,y_0) \in C \setminus A$, we have for any $k \geq K_2$:
    \begin{equation}
        \label{eq:first-iteration-alternate-algo}
        \begin{aligned}
            \|\nabla_x \varphi^k(x_0,y_0) - \nabla_x \varphi\| & \leq \epsilon,\\
            \dist((x_1,y_1), \graph y) & \leq r.
        \end{aligned}
    \end{equation}
	where $\varphi(x) := f(x,y(x))$ and $y = \fni$ for some $i \in \{1,\ldots,N\}$.

	Finally, we construct $\alpha = \min(\alpha_1, \alpha_2, \alpha_3), K = \max(K_1, K_2)$. The main claim will be proved by induction: for all initializations $(x_0,y_0) \in C\setminus A$, for all iterations $\ell \in \NN$, we have:
    \begin{enumerate}[nosep]
        \item $\|x_\ell\| \leq M$.
		\item $x_{\ell}$ is updated to $x_{\ell+1}$ by \eqref{eq:update-inexact-gd} applied to $\varphi$ given in \eqref{eq:first-iteration-alternate-algo} with the function $y$.
        \item $\dist((x_{\ell+1}, y_{\ell+1}), \graph y) \leq r$.
    \end{enumerate}
    \begin{itemize}[leftmargin=*]
        \item \emph{Base case}: by the construction of $K \geq K_2$, we have \eqref{eq:first-iteration-alternate-algo} satisfied and $f(x_0, y(x_0)) \leq l$ (hence, $\|x_0\| \leq M$).
        \item \emph{Induction}: Since $x_{\ell-1}, \ell \geq 1$ is updated by \eqref{eq:update-inexact-gd} applied to $\varphi$, by the construction of $\alpha \leq \alpha_1$, we have $\|x_\ell\| \leq M$. 
        Combining with the induction hypothesis $\dist((x_\ell,y_\ell), \graph y) \leq r$, using \Cref{lemma:remaining-iterations} and the construction of $\alpha_f \leq \alpha_2, K \geq K_1$, $x_{\ell + 1}$ is updated by \eqref{eq:update-inexact-gd}  applied to $\varphi$ and $\dist((x_{\ell+1}, y_{\ell+1}), \graph y) \leq r$. \qedhere
    \end{itemize}
\end{proof}

We now turns to the convergence of \Cref{algo:altenate-method}.
We recall the main convergence result for the inexact gradient descent given in \cite{bolte2024inexact}:
\begin{theorem}[Convergence for inexact gradient method with constant step size {\cite[Theorem 2]{bolte2024inexact}}]
    \label{theorem:convergence-property-inexact-gd}
    Consider a function $f$ $L$-Lipschitz, lower-bounded, semi-algebraic with $\crit_{\epsilon}$ bounded for some $\epsilon > 0$, there is $C > 0, \rho > 0$ such that for any $x_0 \in \RR^p$ and $x_\ell(\alpha)$ generated by inexact gradient descent method with error bounded by $\epsilon$ and fixed step size $\alpha$, we have:
    \begin{equation*}
        \begin{aligned}
            \lim_{\alpha \to 0^+} \limsup_{\ell \to \infty} &\;\; \dist(f(x_\ell(\alpha)), \vcrit_\epsilon f) &&= 0\\
            \limsup_{\alpha \to 0^+} \limsup_{\ell \to \infty} &\;\; \dist(x_\ell(\alpha), \crit f) &&\leq C\epsilon^\rho.\\
        \end{aligned}
    \end{equation*}
\end{theorem}
\begin{remark}
	One needs to argue that in \cite[Lemma 3.2]{bolte2024inexact}, the quantity $\bar{\epsilon}$ can be chosen to be the same as in \cite[Assumption 1]{bolte2024inexact}. This is indeed the case, by coercivity, all arguments take place in a compact set, KL inequality is trivial for non-critical point, and the definable metric subregularity result used is given for an arbitrary compact.
\end{remark}
Combining \Cref{theorem:equivalence-inexact-gd} and \Cref{theorem:convergence-property-inexact-gd}, we obtain the proof for \Cref{theorem:convergence-property-alternate-method}.
\begin{proof}[Proof of \Cref{theorem:convergence-property-alternate-method}]
	Call $P$ the initial distribution over $(x_0,y_0)$ which is assumed to be absolutely continuous. 
	Let $C \subset \RR^n \times \RR^m$ be compact such that $P({(x,y) \in C}) \geq 1 - \frac{\Delta}{2}$. 
	By absolute continuity of the initialization, there is $\Delta'>0$ such that $P(A) \leq \frac{\Delta}{2}$ for any $A$ of Lebesgue measure at most $\Delta'$. We may now consider $A,K,\alpha,M$ as given by \Cref{theorem:equivalence-inexact-gd} (with measure threshold $\Delta'$), so that with probability $1 - \Delta$ over the random choice of $(x_0,y_0)$, the result of \Cref{theorem:convergence-property-inexact-gd} applies.

	In the light of \Cref{theorem:equivalence-inexact-gd}, we only need to verify the assumptions of \Cref{theorem:convergence-property-inexact-gd}. Due to the nature of the main claim, we can always assume that $\alpha_f$ is sufficiently small. For each $i = 1 \ldots, N$, set $f_i \colon x \mapsto f(x,\fn{i}(x))$ as in the proof of \Cref{theorem:convergence-property-inexact-gd}
    \begin{enumerate}
        \item $f_i, i = 1, \ldots, N$ are semi-algebraic: see proof of \Cref{theorem:equivalence-inexact-gd}.
        \item $f_i, i = 1, \ldots, N$ are lower-bounded, with bounded $\crit_{\epsilon}$: see proof of \Cref{theorem:equivalence-inexact-gd}.
        \item By \Cref{theorem:equivalence-inexact-gd}, and with sufficiently small $\alpha_f$, with high probability, the sequence $x_\ell(\alpha_f)$ are uniformly bounded in $B(0,M)$. Since $f_i$ is continuously differentiable, it is Lipchitz when restricted to a compact domain. Thus, with high probability, we can assume that $f_i$ is $L$-Lipschitz, for all $i \in \{1, \ldots, N\}$.
    \end{enumerate}
    By applying \Cref{theorem:convergence-property-inexact-gd}, we conclude the proof.
\end{proof}

\section{Pros and cons of the differentiable programming strategy}
\label{sec:diagonalGradient}
In this section, we study differentiable programming strategy for bilevel gradient method (i.e., \Cref{algo:diagonal-method}), where one minimizes $f(x,\cA^k(x,z))$ as a surrogate for \eqref{eq:original-bilevel-optim}. An investigation of the behavior of this algorithm is of practical interest: for example, standard gradient algorithms to train MAML \cite{originalmaml} for meta learning can be interpreted as \Cref{algo:diagonal-method} applied to the MAML bilevel formulation \cite{implicitMAML}. In addition, replacing the constraint in \eqref{eq:original-bilevel-optim} by a fixed number of GD steps provides a scalable method to deal with \eqref{eq:original-bilevel-optim}, and was employed and studied in the bilevel literature (e.g., \cite[Algorithm 1]{liu2021towards}). 

Certain results in this section are announced with the ``genericity'' notion, similar to \Cref{prop:generic-ae-morse-parametric}. In this section, genericity has the following interpretation. Given $\cF$, a class of functions from $\RR^n$ to $\RR$, if a property $P$ is generic in $\cF$, then for any $f$, the set $\{a \in \RR, g(\cdot) := f(\cdot) + \innerproduct{a}{\cdot} \text{ satisfies } P\}$ is open, dense and of full Lebesgue measure. The term ``for generic $f$'' can therefore be understood as ``for any $f$ up to a typical linear perturbation'', see for instance \Cref{lemma:sard}.

The purpose of this section is to compare the optimization landscape of bilevel optimization problem \eqref{eq:original-bilevel-optim} and the minimization problem
\[
{\min} \;\; \varphi^k(\firstvar,\thirdvar) := f(\firstvar, \algo^k(\firstvar,\thirdvar))\mbox{ for $k\in\NN$.}
\]
As we cannot formally take $k=\infty$, this approach is connected to 
the unconstrained version of problem \eqref{eq:original-bilevel-optim} which we refer to as the single-level problem
\begin{equation}
    \label{eq:singleLevel}
    \tag{SL}
    \begin{aligned}
        \underset{\firstvar \in \RR^n,\secondvar \in \RR^m}{\min} & \quad f(\firstvar,\secondvar) .
    \end{aligned}
\end{equation}

\subsection{Equivalence with the unconstrained single level problem}
A fundamental property of the gradient descent map is that it defines a \emph{global diffeomorphism} of the space.
The following result is well known in optimization, we provide a proof for completeness.
\begin{lemma}[Gradient descent defines a global diffeomorphism]
	Under \Cref{assumption:lowerLevel} (\TL{except semi-algebraicity of $f$ and $g$}), the mapping $(x,z) \mapsto (x,\algo^k(x,z))$ is a global diffeomorphism and the inverse is given by $k$ applications of
	\begin{align*}
		\algo^{-1}(x, \cdot)  = \prox_{-\alpha g_x}
	\end{align*}
	which is well-defined and single valued.
	\label{lem:proxInverse}
\end{lemma}
\begin{proof}
	For $k = 1$, and for any $x \in \RR^n$, the Jacobian of $\algo(x,\cdot)$ at $\thirdvar \in \RR^m$ is given by
	\begin{align*}
		\jacnospace_z \algo(x,z) = \bI - \alpha_g \nabla^2_{yy} g(x,z)	
	\end{align*}
	and by assumption $\alpha_g \|\nabla^2_{yy} g(x,z)\|_{\mathrm{op}} < 1$, so that this matrix is invertible and $\algo(x,\cdot)$ is a local diffeomorphism. It is actually a global diffeomorphism since 
	\begin{align*}
		\algo(x,y) = \algo(x,z) \qquad 
			&\implies \qquad  \|y - z\| =  \|\alpha_g( \nabla_y g(x,z) - \nabla_y g(x,y))\| \leq \alpha_g L_g \|y-z\| \\
			&\implies \qquad  \|y - z\|  (1 - \alpha_g L_g)  \leq 0 \\
			&\implies \qquad  \|y - z\| =  0,
	\end{align*}
	so that it is injective. Regarding the prox characterization, we have
	\begin{align*}
		0 = z-u - \alpha_g \nabla_y g(x,z) \qquad \iff \qquad z = \argmin_v \ \frac{1}{2} \|v - u\|^2 - \alpha_g g(x,v).
	\end{align*}
	The result for general $k$ follows by composition.
\end{proof}

\begin{proposition}[Equivalence with the single level problem]
    \label{prop:equivalence-single-level}
	Under \Cref{assumption:lowerLevel} (\TL{except semi-algebraicity of $f$ and $g$}), for any $k \in \NN$, up to a diffeomorphism, the critical points of $f$ are the same as those of $\varphi^k$, with the same objective value and the same index (local minimum/maximum, second order structure).
\end{proposition}
\begin{proof}
	Since $\varphi^k(x,z) = f(x, \algo^k(x,z))$ the equivalence follows from \Cref{lem:proxInverse}. We have
	\begin{align*}
		\nabla \varphi_k(x,z) &= 
		V_k^\top \nabla f(x, \algo^k(x,z)) 
	\end{align*}
    where
    \begin{align*}
        V_k = 
        \begin{pmatrix}
			I & 0 \\
			\jacnospace_x \algo^k(x,z) & \jacnospace_z \algo^k(x,z)
		\end{pmatrix},
    \end{align*}
	and the left matrix is invertible so that $\nabla \varphi_k(x,z) = 0$ if, and only if, $\nabla f(x, \algo^k(x,z)) = 0$. The local minimum / maximum structure is preserved by diffeomorphisms. As for the second order structure, for any critical point of $\varphi^k$, we have
	\begin{align*}
		\nabla^2 \varphi^k(x,z) = 
		V_k^\top 
		\nabla^2 f(x, \algo^k(x,z)) 
        V_k,
	\end{align*}
	because $\nabla f(x, \algo^k(x,z)) = 0$. Therefore, the Hessian of $\varphi^k$ is congruent to that of $f$ and they have therefore the same number of positive, negative, and null eigenvalues.
\end{proof}
{Note that due to the construction in the previous proof, for any $k\in\NN$, we have: $(x,y) \in \crit\, \varphi^k$ if and only if $(x, \cA^{-k}(x,z)) \in \crit f$ (and the function $\cA^{-k}(x,\cdot)$ is given as in \Cref{lem:proxInverse}).}

\subsection{Pseudo-stability of the differentiable programming strategy}

\Cref{prop:equivalence-single-level} is a rather negative result as it entails that there is no qualitative difference between $\varphi^k$ and $f$ in terms of their critical points. In other words, the approximate value function $\varphi^k$ \emph{ignores} the bilevel constraint. On the other hand, the goal of this section is to establish that near a local minimizer the algorithm has a pseudo-stability property: if the iterates meet a solution neighborhood, they tend to remain close for a long time. Although we describe an exponentially long stability duration, it is different from genuine stability as we empirically illustrate in \Cref{sec:illustration}, hence the term ``pseudo-stability''.

\smallskip

Let us recall that, given a $C^2$ function $h:\RR^n\to\RR$, the point $x^{\star} \in \RR^m$ is a strong local minimizer of $h$, if
    $$\nabla h(x^\star) = 0, \text { and } \nabla^2 h(x^\star) \succ 0.$$
    
    Assume $y^\star$ is a strong local minimizer for $g_{x^\star}$ for some couple $(x^\star,y^\star)$. Observe that, by the implicit function theorem, there is a local map $\bar{y}$ such that $\bar{y}(x)$ is a strict local minimum of $g_x$ in a neighborhood of $x^\star$.
    
    When $y$ is a $C^2$ curve which satisfies $y(x)\in \crit g_x$, and $x^\star$ is  strong local minimizer of $\varphi \colon x \mapsto f(x,y(x))$, {\em the pair $(x^\star,y^\star)$} is called a \emph{strong  local minimizer} of problem \eqref{eq:original-bilevel-optim}. 
    In particular, if one assumes \ref{assumption:constraint-qualification}, this describes a strong local minimizer of $f(x,\fn{i}(x))$ with $i \in \{1, \ldots,N\}$ (the quantities $N$, $\fn{i}$ being defined as in \Cref{prop:crit-local-structures}).

\begin{theorem}[Pseudo-stability around strong local minimizers]
	Let \Cref{assumption:lowerLevel} (\TL{except semi-algebraicity of $f$ and $g$}) hold and let $(\firstvar^\star,\thirdvar^\star)$ be a strong local minimum of \eqref{eq:original-bilevel-optim}.

	Then, there exists $\delta>0$ and $\gamma>0$, as well as constants $C>0$ and $\TL{0<\rho <1}$, such that for all $0 < \alpha_f \leq \gamma$, for all $k \in \NN$, \Cref{algo:diagonal-method}, initialized at $(x,z) \in \RR^n \times \RR^m$ such that $\|x - x^\star\| + \|z-z^\star\| < \delta$ satisfies $\|x_\ell- x^\star\| + \|z_\ell - z^\star\| \leq \delta$, for all
	\begin{align*}
		\ell \leq \frac{\delta - \|x - x^\star\| - \|z - z^\star\|}{\alpha_f C \rho^k}.
	\end{align*}
	\label{th:stabilityLocalMin}
\end{theorem}

\begin{proof}[Proof of \Cref{th:stabilityLocalMin}]
	Let $U$, $\rho$ and $C$ be given by \Cref{lemma:local-convexity}.
	Since $g$ is $C^3$, the implicit function ${y}$ of the function $\nabla_y f(x,y) = 0$ is $C^2$. As a consequence, $\varphi(x) := f(x,y(x))$ is $C^2$. We may pick $\gamma>0$ such that $\gamma \|\nabla^2 \varphi(x)\| < 2$ uniformly on $U$. Furthermore, since $x^\star$ is a strict second order minimizer of $\varphi$ we may pick $\delta>0$ such that $\nabla^2 \varphi(x) \succeq 0$ for all $x$ such that $\|x- x^\star\| \leq \delta$, and such that $B(x^\star,\delta) \times B(z^\star,\delta) \subset U$.

	The function $\varphi$ is strongly convex on $B$ with a global minimizer at $x^\star$. This implies that the gradient mapping $x \mapsto x - \alpha \nabla \varphi(x)$ is a contraction on $B$ with global fixed point at $x^\star$. From the update of \Cref{algo:diagonal-method}, we obtain for all $\ell \geq 1$,
	\begin{align*}
		\|z_{\ell} - z^\star\| &= \|z_{\ell-1} - z^\star - \alpha_f \nabla_z \varphi^k(x_{\ell-1},z_{\ell-1})\| \leq \|z_{\ell-1} - z^\star\| + \alpha_f C \rho^k \leq\|z - z^\star\| + \ell \alpha_f C \rho^k, \\
	\end{align*}
    since $\|\nabla_z \varphi^k(x_{\ell-1},z_{\ell-1})\| \leq C\rho^k$ for some constants $C, \rho > 0$ in the neighborhood $U$ (cf. \Cref{lemma:local-convexity}). Moreover,
    \begin{align*}	
        \|x_{\ell} - x^\star\| &= \|x_{\ell-1} - x^\star - \alpha_f \nabla_z \varphi^k(x_{\ell-1},z_{\ell-1})\| \\
		&\leq  \|x_{\ell-1} - x^\star - \alpha_f \nabla \varphi(x_{\ell-1})\| +  \alpha_f \|\nabla \varphi(x_{\ell-1})- \nabla_x \varphi^k(x_{\ell-1},z_{\ell-1})\|\\  
		&\leq \|x_{\ell-1} - x^\star\| + \alpha_f C \rho^k\\ 
		&\leq \|x - x^\star\| + \ell\alpha_f C \rho^k,
	\end{align*}
    again, by using the estimation $\|\nabla \varphi(x_{\ell-1})- \nabla_x \varphi^k(x_{\ell-1},z_{\ell-1})\| \leq C\rho^k$ (cf. \Cref{lemma:local-convexity}).
	We conclude that if $\|z_{\ell} - z^\star\| + \|x_{\ell} - x^\star\| > \delta$, one has 
	\begin{align*}
		\|x - x^\star\| + \|z - z^\star\| + 2\ell\alpha_f C \rho^k \geq \delta
	\end{align*}
	and the result follows (ignoring the constant $2$).
\end{proof}

The following is a consequence of the Morse-Sard theorem (see proof in \cite[Chapter 1.7]{guillemin2010differential}).
\begin{lemma}[Linear perturbation of a semialgebraic function]
	\label{lemma:sard}
	Let $f,g$ be as in \Cref{assumption:lowerLevel} (\TL{except semi-algebraicity of $f$ and $g$}) and let $\bar{y} \colon \RR^n \to \RR^m$ be $C^2$.
	Generically, for $a \in \RR^n$, the function $x \mapsto  f(x,\bar{y}(x)) + \left\langle x, a \right\rangle$ is Morse.
\end{lemma}

Combining \Cref{lemma:sard} with \Cref{th:stabilityLocalMin}, we obtain the following corollary.
\begin{corollary}
	Under \Cref{assumption:lowerLevel}, for a generic $f$ and a parametric Morse $g$, the result of \Cref{th:stabilityLocalMin} holds for all local minimizers of the bilevel problem \eqref{eq:original-bilevel-optim}.
	\label{cor:morseParametericStability}
\end{corollary}
The consequence of \cref{cor:morseParametericStability} is that there are neighborhoods of local minimizers of the original bilevel program \eqref{eq:original-bilevel-optim} which are pseudo-stable for the differentiable programming approximation, minimized by \Cref{algo:diagonal-method}, for large $k$. This means that for very large $k$ if the recursion visits this neighborhood, it will not \TL{necessarily} remain in it for an infinitely long time, but it will stay in the neighborhood for a time which is exponential in $k$. \TL{In particular, the iterates will eventually leave this neighborhood unless this region contains critical points of $\varphi^k$ (indeed, since $\alpha \leq \gamma < 2 / L$ where $L$ is the Lipchitz constant of gradient in the neighborhood), we are in the small step-size regime where the objective function decreases monotonically and any accumulation point of the iterates is a critical point).} 
To appreciate this result, it is important to remark that in \Cref{th:stabilityLocalMin}, the quantities $\gamma$ which bounds step size $\alpha$ which defines the outer recursion, the radius $\delta$, the constants $C$ and $\rho$ do not depend on $k$.

\subsection{Repulsivity of unconstrained critical points}

The main result of this section  states that under favorable assumptions the critical points of $\varphi^k$ can actually be reached under specific and unlikely circumstances.

\begin{theorem}[Escape at infinity and sharpness]
    \label{th:escapeInfinitySharpness}
	Under \Cref{assumption:lowerLevel}, consider $(\firstvar^\star,\secondvar^\star) \in \RR^n \times \RR^m$ satisfying $\nabla f(\firstvar^\star,\secondvar^\star) = 0$ and such that $\nabla^2 f(x^\star,y^\star)$ invertible. Assume $\secondvar^\star$ is not a local minimum of $g_{\firstvar^\star}$ and that $g_{\firstvar^\star}$ is a Morse function. Consider for each $k \in \NN$ the unique $y^{-k} \in \RR^m$ such that $\algo^k(x^\star,y^{-k}) = y^\star$ as given in \Cref{lem:proxInverse}.  
	We have the following alternatives:
    \begin{enumerate}[nosep]
		\item $\lim_{k\to \infty} \|y^{-k}\| = + \infty$.
		\item There exists $C > 0$, $\rho > 1$ and $k_0 \in \NN$, such that: $\|\nabla^2 \varphi^k(\TL{\firstvar^\star},\secondvar^{-k})\|_{\mathrm{op}} \geq C\rho^{2k}$, for all $k \geq k_0$.
    \end{enumerate}
\end{theorem}

Before proving \Cref{th:escapeInfinitySharpness}, we explain its implication: \TL{it is difficult to converge to the critical points of $\varphi^k$ that are not related to the original bilevel problem. Indeed, we remind the readers that $(x^\star, y^\star)$ is a critical point of $f$ if and only if $(x^\star, y^{-k})$ is a critical point of $\varphi^k$. If $y^\star$ is not a local minimum of $g_{x^\star}$ (i.e., $(x^\star, y^{-k})$ is a ``fake'' critical point of $\varphi^k$), then \Cref{algo:diagonal-method} (which is equivalent to gradient desent on $\varphi^k$ with step-size $\alpha_f$) either:
\begin{enumerate}
    \item Takes a very long time to reach these ``fake'' critical points since $y^{-k}$ diverges to infinity when $k \to \infty$ (the first possibility in \Cref{th:escapeInfinitySharpness}). 
	\item Converges to a critical point with exponentially large curvature (w.r.t. $k$) (the second possibility in \Cref{th:escapeInfinitySharpness}). This requires either an exponentially small step size (i.e., $\alpha_f = o(1/\rho^{2k})$ for some $\rho > 1$), or constitutes a very unlikely event since in this case the Jacobian is not a contraction in this case (see \textit{e.g.} \cite{shub1987global} or more precisely \cite[Theorem 2.1 or Proposition 3.1]{bolte2025convergence}). Given the usual range of learning rates $\alpha_f$ in practical settings such as deep learning, \Cref{algo:diagonal-method} is unlikely to converge to these points if $k$ is sufficiently large.
\end{enumerate}
}
We will need the following lemma on operator norms.
\begin{lemma}
    \label{prop:operator-bound}
	Given two matrices $A \in \RR^{n \times m}, B \in\RR^{n \times n}$, with $B$ symmetric. We have
    \begin{equation*}
		\|A^\top B A\|_\mathrm{op} \geq \|A\|_\mathrm{op}^2\lambda_{\text{min}}(B).
    \end{equation*}
    where $\lambda_{\text{min}}(B) > 0$ is the smallest absolute value of eigenvalues of $B$.
\end{lemma}
\begin{proof}
	We have for any $x$, $\|x^\top A^\top B A x\| \geq \lambda_{\text{min}}(B) \|A x\|_2^2$,
	so that 
    \begin{align*}
		\|A^\top B A\|_\mathrm{op}  = \max_{\|x\|=1} \|x^\top A^\top B A x\| \geq \max_{\|x\|=1}  \lambda_{\text{min}}(B) \|A x\|_2^2 = \|A\|_\mathrm{op}^2\lambda_{\text{min}}(B). 
    \end{align*}
\end{proof}

\begin{proof}[Proof of \Cref{th:escapeInfinitySharpness}]
	Thanks to \Cref{lem:proxInverse}, one can interpret the sequence $\{y^{-k}\}_{k \in \NN}$ as a proximal point sequence to minimize $- g(\firstvar^\star, \cdot)$, that is, to maximize $g(x^\star,\cdot)$, initialized at $\secondvar^\star$. By \cite{attouch2013Convergence}, there are exactly two alternatives:
    \begin{enumerate}
		\item $\lim_{k\to \infty} \|y^{-k}\| = + \infty$.
		\item $\lim_{k\to \infty} y^{-k} = \bar{y}$ where $\bar{y}$ is a critical point of $g(x^\star,\cdot)$. 
    \end{enumerate}
	Let us consider the second case. 
	The point $\bar{y}$ is not a local minimizer of $g(\firstvar^\star,\cdot)$. Indeed, if $\bar{y}$ is a local minimizer of $g(\firstvar^\star,\cdot)$, since $-\alpha_g g(x^\star,\cdot) + \frac{1}{2}\|y^\star-\cdot\|^2$ attains its global minimum at $y^\star$  (because $\alpha_g L_g<1$), the proximal iterates are therefore stationary so that $y^\star = \bar{y}$ is also a local minimizer of $g(x^\star,\cdot)$. This last case is ruled out by our hypotheses.

	We recall that from the proof of \Cref{prop:equivalence-single-level}, since $\nabla f(\firstvar^\star,\secondvar^\star) = 0$, we have 
	\begin{align*}
		\nabla^2 \varphi^k(x^\star,y^{-k})
		=
		V_{-k}^\top 
		\nabla^2 f(x^\star, \TL{y^\star}) 
		V_{-k}
	\end{align*}
    where
    \begin{align*}
        V_{-k} = \begin{pmatrix}
			I & 0 \\
			\jacnospace_x \algo^k(x^\star,y^{-k}) & \jacnospace_z \algo^k(x^\star,y^{-k})
		\end{pmatrix}.
    \end{align*}
	We are going to show that, for large $k$, $\jacnospace_z \algo^k(x^\star,y^{-k})$ has an exponentially large operator norm and the result will follow from \Cref{prop:operator-bound}. We recall from the proof of \Cref{lemma:local-convexity} the following
    \begin{equation*}
		\partial_{\thirdvar} \algo^k(\firstvar^\star,y^{-k}) = \prod_{i=1}^{k} \underbrace{\left(\bI - \alpha_g \partialss^2 g(\firstvar^\star,y^{-i})\right)}_{\bH_i},
    \end{equation*}
    
    Let $\bH_{a,b} = \prod_{i = a}^b \bH_i$, it is sufficient to prove that $\|\bH_{k_0,k}\| \geq K\rho^{k-k_0}$ for some $\rho > 1, K > 0$ and $k_0 \in \NN$. It is true because $\bH_i$ are symmetric, positive definite with $\lambda_{\min}(\bH_i) \geq 1 - \TL{\alpha_g}L_g > 0$ (cf. \Cref{assumption:lowerLevel}) for all $i \in \NN$. 

    Since $\partialss^2 g(\firstvar^\star,\bar{\secondvar})$ is symmetric, there exists an orthogonal matrix $\bP$ such that $\bP^\top \partialss^2 g(\firstvar^\star,\bar{\secondvar}) \bP$ is diagonal. We conclude that the same matrix $\bP$ also makes:
    \begin{equation*}
        \bD = \bP^\top (\bI - \alpha_g \partialss^2 g(\firstvar^\star,\bar{\secondvar}))\bP
    \end{equation*}
    diagonal. Moreover, due to the Morse assumption of $g(\firstvar^\star, \cdot)$ and its Lipchitz gradient
    (cf. \Cref{assumption:lowerLevel}), the matrix $\bD$ contains two sets of eigenvalues: $\{\lambda \mid \lambda > 1\}$ and $\{\lambda \mid 0 < \lambda < 1\}$. Without loss of generality, we can suppose that:
    \begin{equation*}
        \bD = \begin{pmatrix}
            \bD_1 & \mbf{0}\\
            \mbf{0} & \bD_2
        \end{pmatrix},
    \end{equation*}
    where $\bD_1$ (resp. $\bD_2$) contains only positive elements that are bigger (resp. smaller) than $1$. Note that the size of $\bD_1$ is at least one because $\partialss^2 g(\firstvar^\star,\bar{\secondvar})$ has at least one negative eigenvalue.
    
    Since $\lim_{k \to \infty} \secondvar^{-k} = \bar{\secondvar}$ and $g$ is $C^2$ (cf. \Cref{assumption:lowerLevel}), for any $\epsilon > 0$, one can choose $k_0 \in \NN$ such that for $\ell \geq k_0$, one can write: $\partialss^2 g(\firstvar^\star,\secondvar^{-\ell}) = \partialss^2 g(\firstvar^\star,\bar{\secondvar}) + \bE_{\ell}$ where $\|\bE_{\ell}\| < \epsilon / \alpha_g$. Thus, 
    \begin{equation*}
        \bP^\top (\bI - \alpha_g \partialss^2 g(\firstvar^\star,\secondvar^{-\ell}))\bP = \bP^\top (\bI - \alpha_g \partialss^2 g(\firstvar^\star,\bar{\secondvar}) - \alpha_g \bE_\ell)\bP = \bD - \underbrace{\alpha_g\bP^\top\bE_\ell\bP}_{\bE_\ell'}.
    \end{equation*}
    Since the matrix $\bP$ is orthogonal, $\|\bE_\ell'\| = \alpha_g\|\bE_\ell\| \leq \epsilon$. Thus,
    \begin{equation*}
        \begin{aligned}
            \|\bH_{k_0,k}\| &= \left\|\prod_{\ell = k_0}^k \left(\bI - \alpha_g \partialss^2 g(\firstvar^\star,\secondvar^{-\ell})\right)\right\|\\
            &= \left\|\bP^\top\left[\prod_{\ell = k_0}^k \left(\bI - \alpha_g \partialss^2 g(\firstvar^\star,\secondvar^{-\ell})\right)\right] \bP\right\|\\
            &= \left\|\prod_{\ell = k_0}^k \bP^\top\left(\bI - \alpha_g \partialss^2 g(\firstvar^\star,\secondvar^{-\ell})\right)\bP \right\|\\
            &= \left\|\underbrace{\prod_{\ell = k_0}^{k} (\bD + \bE_{\ell}')}_{\bH}\right\|\\
        \end{aligned}
    \end{equation*}
    In the following, WLOG, we will assume $k_0 = 1$ and prove $\|\bH_{k_0,k}\| = \|\bH\| \geq K\rho^{k}$. 
    
    Consider the product between $\bH$ and $e_1$ - the first canonical basis in $\RR^n$. Define the sequence of pairs of vectors $\{(u_n, v_n)\}$ whose sizes are equal to the number of elements bigger than and smaller than $1$ of $\bD$ respectively as follows:
    \begin{equation*}
        \begin{pmatrix}
            u_0 \\
            v_0
        \end{pmatrix} = e_1, \qquad \begin{pmatrix}
            u_{\ell+1} \\ v_{\ell+1}
        \end{pmatrix} = \left(\bD + \bE'_{\ell}\right)\begin{pmatrix}
            u_\ell \\ v_\ell
        \end{pmatrix}, \forall \ell \geq 0.
    \end{equation*}
    In particular, we have:
    \begin{equation*}
        \bH e_1 = \begin{pmatrix}
            u_{k} \\ v_k
        \end{pmatrix}
    \end{equation*}
    Moreover, if one write $\bE'_\ell$ as a block matrix:
    \begin{equation*}
        \bE_\ell' = \begin{pmatrix}
            \bE^{1,1}_\ell & \bE^{1,2}_\ell \\
            \bE^{2,1}_\ell & \bE^{2,2}_\ell
        \end{pmatrix}
    \end{equation*}
    we get the following relation:
    \begin{equation*}
        \begin{aligned}
            u_{\ell+1} &= (\bD_1 + \bE^{1,1}_\ell) u_{\ell} + \bE^{1,2}_\ell v_\ell\\ 
            v_{\ell+1} &= \bE^{2,1}_\ell u_{\ell} + (\bD_2 + \bE^{2,2}_\ell) v_\ell
        \end{aligned}
    \end{equation*}
    Denote $\tau_1 > 1 > \tau_2$ the smallest (resp. the largest) non-zero elements of $\bD_1$ and $\bD_2$, we obtain an estimation on the norm of $u_\ell, v_\ell$, using the fact that $\|\bE_\ell^{i,j}\| \leq \|\bE_\ell'\| \leq \epsilon, \forall i, j \in \{1,2\}$:
    \begin{equation*}
        \begin{aligned}
            \|u_{\ell+1}\| &\geq (\tau_1 - \epsilon) \|u_{\ell}\| - \epsilon \|v_\ell\|\\ 
            \|v_{\ell+1}\| &\leq \epsilon \|u_\ell\| + (\tau_2 + \epsilon)\|v_\ell\|
        \end{aligned}
    \end{equation*}
    Therefore, if $\epsilon$ is small enough so that $\tau_1 - \epsilon > 1 > \tau_2 + \epsilon$, we get:
    \begin{equation*}
        \begin{aligned}
            \|u_{\ell+1}\| &\geq (\tau_1 - \epsilon) \|u_{\ell}\| - \epsilon \|v_\ell\|\\
            &\geq (\tau_1 - \epsilon) \|u_{\ell}\| - \epsilon^2 \|u_{\ell - 1}\| - \epsilon (\epsilon + \tau_2) \|v_{\ell-1}\|\\
            &\geq (\tau_1 - \epsilon) \|u_{\ell}\| - \epsilon^2 \|u_{\ell - 1}\| - \epsilon^2 (\epsilon + \tau_2) \|u_{\ell-2}\| - \epsilon (\epsilon + \tau_2)^2 \|v_{\ell-2}\|\\
            &\ldots \\
            &\geq (\tau_1 - \epsilon) \|u_{\ell}\| - \epsilon^2\left(\sum_{j = 0}^{\ell - 1} (\tau_2 + \epsilon)^{\ell - 1 - j} \|u_j\|\right) \\
            &\geq (\tau_1 - \epsilon)\|u_\ell\| - \frac{\epsilon^2}{1 - \tau_2 - \epsilon} \max_{j = 0, \ldots, \ell - 1} \|u_{j}\|
        \end{aligned}
    \end{equation*}
    For $\epsilon$ sufficient small so that $\rho := \tau_1 - \epsilon - \epsilon^2 / (1 - \tau_1 -\epsilon) > 1$, we get $\|u_{\ell}\| \geq \rho^\ell, \forall \ell \leq 0$. Hence, $\|\bH_{k_0,k}\| = \|\bH\| \geq \|u_k\| \geq \rho^k$, and the result follows.
\end{proof}

\begin{corollary}
	Under \Cref{assumption:lowerLevel}, for a generic $f$ and $g$, the result of \Cref{th:escapeInfinitySharpness} holds for all critical points of the single level problem \eqref{eq:singleLevel}.
	\label{cor:instability}
\end{corollary}

\TL{
We finish this section with another corollary, which is an application of \Cref{th:escapeInfinitySharpness} when we have in addition the parametric Morse property (cf. \Cref{def:parametric_morse}) for $g$.
\begin{corollary}
    Consider bilevel optimization where \Cref{assumption:lowerLevel} holds, $g$ is parametric Morse and $f$ is a generic function. We have: for any bounded set $K \subseteq \RR^n \times \RR^m$, there exists $k_0 \in \NN$ such that for any $k \geq k_0$, there are three possibilities regards a critical point $(x,y)$ of $\varphi^k$:
    \begin{enumerate}
        \item $\cA^k(x,y)$ is a local minimum of $g(x,\cdot)$.
        \item $(x,y) \notin K$.
        \item $\|\nabla^2 \varphi^k(x,y)\| \geq C\rho^{2k}$ for some constants $C > 0, \rho > 1$. In addition, there exists a neighborhood $U$ of $(x,y)$ where $\|\nabla^2 \varphi^k(\tilde{x},\tilde{y})\| \geq C\rho^{2k} / 2$ for all $(\tilde{x},\tilde{y}) \in U$.
    \end{enumerate}
\end{corollary}
\begin{proof}
    This is just an application of \Cref{th:escapeInfinitySharpness}. Indeed, the fact that $g$ is parametric Morse ensures the condition $g_x$ being Morse as in \Cref{th:escapeInfinitySharpness}. Three possibilities then follow correspondingly. Lastly, the neighborhood $U$ in the last claim exists because $\varphi^k$ is a $C^2$ function (remember that $g$ and $f$ are $C^2$ and $C^3$ respectively as assumed in \Cref{assumption:lowerLevel}). 
\end{proof}
}

\subsection{Illustrations of theoretical findings}
\label{sec:illustration}
We illustrate the findings of this section in \Cref{fig:numericalIllustration,fig:numericalIllustration2} with a simple example. Let $h \colon \RR \to \RR$ denotes the Huber loss, $h(t) = t^2/2$ for $t \in [-1,1]$ and $h(t) = |t| -1/2$ for $|t| > 1$. This function is $C^1$ and has vanishing curvature outside $[-1,1]$ which allows us to compose it with a square while preserving the global Lipschitz-continuity of the gradient. The following examples do not contradict \cite[Theorem 3.2] {Franceschi2018BilevelPF}  which requires both uniform convergence of the lower level recursion to a unique argmin, and a bounded set of initialization. 

\paragraph{Illustrating instabilities:} We consider the bilevel problem as in \eqref{eq:original-bilevel-optim} with 
\begin{equation}
    \label{eq:example-bilevel-optim}
    \begin{aligned}
        \underset{\firstvar \in \RR^n,\secondvar \in \RR}{\min} & \quad f(\firstvar,\secondvar):= (\secondvar - 0.5)^2\qquad \qquad
        \text{s.t.} &\quad \secondvar \in  {\argmin_\thirdvar}\; g(\firstvar,\thirdvar) := h(z^2-1)
    \end{aligned}
\end{equation}
where the two functions
$f(x,y) = (y - 0.5)^2$ and $g(x,y) = h(y^2-1)$, are displayed in \Cref{fig:numericalIllustration}. This is a degenerate problem as there is no dependency in $x$, it corresponds to the following trivial problem, yet, it illustrates our results very well.  
\begin{equation}
    \begin{aligned}
        \underset{\secondvar \in \RR}{\min} & \quad  (\secondvar - 0.5)^2\qquad\qquad \qquad
        \text{s.t.} &\quad \secondvar = \pm 1.
    \end{aligned}
\end{equation}

The inner gradient descent recursion in \eqref{eq:algorithmic-bilevel} is attracted by $\pm 1$, with a smooth dependency on the initialization around $\pm 1$, and a sharp transition around $0$, the local maximizer of $g$. The higher the value of $k$, the steeper this transition. As a consequence, when composing $f$ and $\algo^k$, the landscape of $f$ is completely modified. It has two very flat regions where the derivative almost vanishes in a large neighborhood around $\pm 1$, and a sharp global minimum, corresponding to the global minimum of $f$. These are illustrated in \Cref{fig:numericalIllustration} where we represent $\algo^k$ and $f \circ \algo^k$ for $k = 9$. The larger $k$, the flatter the flat areas, and the sharper the global minimizer. 
\begin{figure}[H]
	\centering
\includegraphics[width=\textwidth]{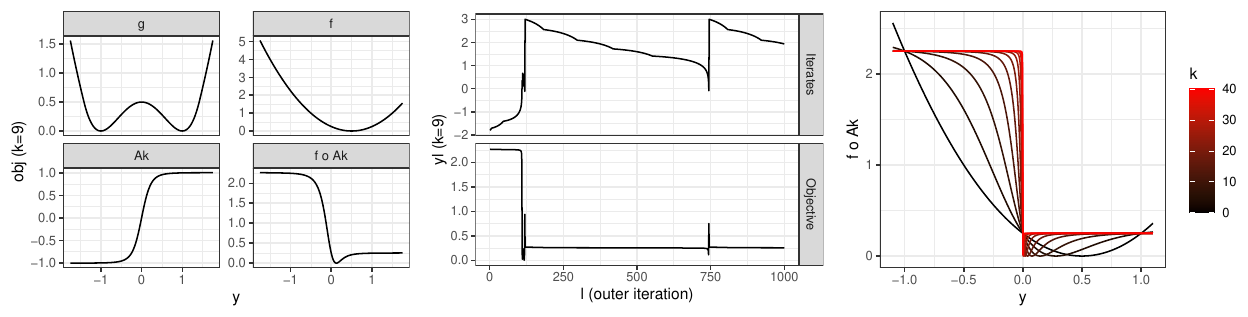}
	\caption{Illustration of the instability phenomenon of the diagonal method. Left: inner objective $g$ and outer objective $f$, with the outcome of $k=9$ gradient steps on $g$, depending on the initialization $\algo^k(z)$ and the corresponding approximation $\varphi^k(z) = f(\algo^k(z))$. Middle: {values of $y_l$ and $\varphi^k(y_l)$} along outer iterations. Many iterations have objective values corresponding to solutions of the bilevel problem \eqref{eq:original-bilevel-optim} but tends to be attracted by the sharp global minimizer of $f$, which is repulsive for \Cref{algo:diagonal-method}. Right: profile of $\varphi^k = f \circ \algo^k$ for different values of $k$. {The \emph{sharp} minimizer of $\varphi^k$ at $0$ corresponds to the second situation of \Cref{th:escapeInfinitySharpness}.}}
	\label{fig:numericalIllustration}
\end{figure}
Now running the gradient recursion on $f \circ \algo^k$ as in \Cref{algo:diagonal-method}, we obtain the following behavior, illustrated in \Cref{fig:numericalIllustration}, middle. The recursion spends most of its iterations around the large flat areas. It does not stay there however since they do not contain stationary points, it is attracted by the sharp minimizer which is a repulsive point of the recursion because of the high curvature. If one increases the value of $k$, the behavior remains qualitatively similar, with flatter and more curved areas as illustrated in \Cref{fig:numericalIllustration} right. Actually, in this case, $\varphi^k$ converges pointwise to $f(-1) = 2.25$ for $y<0$, $f(1) = 0.25$ for $y>0$ and $f(0) = 0.25$ for $y=0$. However, $\min \varphi^k = 0$ for all $k$ and the argmin converges to $0$, the local maximum of $g$. This is an example of pointwise, but non uniform convergence, and the limit is indeed discontinuous.
\begin{figure}[H]
	\centering
	\includegraphics[width=\textwidth]{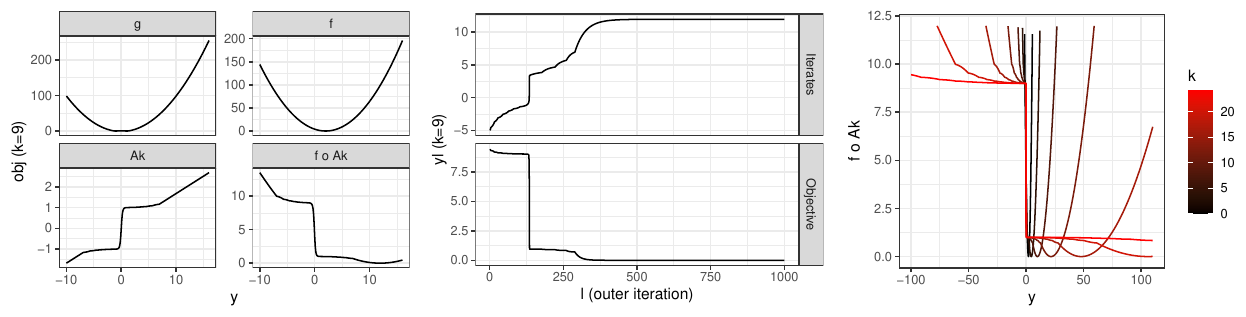}
	\caption{Same as \Cref{fig:numericalIllustration} with $f(x,y)= (y - 2)^2$. We see that the recursion spends time with value $\varphi^k(y_l) \sim f(-1)$ and then converges to a point corresponding the global minimum of $f$. The specificity of this setting is that when the inner iteration counter $k$ is increasing, the corresponding {minimizer} of $\varphi^k = f \circ \algo^k$ is pushed to infinity, {which corresponds to the first situation of \Cref{th:escapeInfinitySharpness}. Note that the functions $g$ and $\algo^k$ are the same as in \Cref{fig:numericalIllustration}. Since we are illustrating the escape at infinity, the abscissa is much larger so that the variations around the origin, visible in \Cref{fig:numericalIllustration} are squeezed here.}
    }
	\label{fig:numericalIllustration2}
\end{figure}

\paragraph{Illustrating escape at infinity:} We consider the same experiment with $f(x,y)= (y - 2)^2$ instead. This time the global minimizer of $f$ is $2$ which {lies to the right of all critical points of} $g$. The inner gradient descent steps are still attracted by $\pm 1$. But for a given $k$, it is possible to find $\algo^k(y) \simeq 2$ by choosing $y$ large enough. For large $k$, the corresponding initialization $y$ diverges. Furthermore, this corresponds to a region of positive curvature for $g$ and it does not generate the sharp minimizer behavior for $f \circ \algo^k$, which we had previously. It modifies the landscape of $f$, pushing smoothly the global minimizer to infinity. In this case, the gradient descent algorithm applied to $f \circ \algo^k$ finds the global minimum of $f$, ignoring the bilevel constraints as illustrated in \Cref{fig:numericalIllustration2}. In this case $\varphi^k$ also converges pointwise to a discontinuous function.

\section*{Acknowledgements}
JB, TL, EP thank AI Interdisciplinary Institute ANITI funding, through the ANR under the France 2030 program (grant ANR-23-IACL-0002), Chair TRIAL, Air Force Office of Scientific Research, Air Force Material Command, USAF, under grant numbers FA8655-22-1-7012. JB, EP and SV acknowledge support from ANR MAD. JB and EP thank TSE-P and acknowledge support from ANR Chess, grant ANR-17-EURE-0010, ANR Regulia. EP acknowledges support from IUF. SV thanks PEPR PDE-AI (ANR-23-PEIA-0004) and the chair 3IA BOGL (ANR-23-IACL-0001). TL is supported by the French ANR through the MIAI Cluster (ANR-23-IACL-0006).
{The authors warmly thank S. Dempe and A. Zemkoho for their useful comments.}

\bibliographystyle{plain}
\bibliography{references}

\end{document}